\documentclass[11pt]{amsart}
\usepackage[font=small,skip=0pt]{caption}
\usepackage{setspace}
\usepackage{amsmath, amsfonts, amsthm, amstext, xspace, mathabx}
\usepackage[mathscr]{euscript}
\usepackage{amssymb,latexsym}
\usepackage{comment}
\usepackage{xcolor}
\usepackage[colorlinks=true,backref=page]{hyperref}
\usepackage{soul}
\usepackage{longtable,booktabs}
\definecolor{seaolive}{RGB}{46,139,87}

\usepackage{cleveref}
\definecolor{seagreen}{RGB}{46,139,87}
\definecolor{maroon}{RGB}{128,0,0}
\definecolor{darkviolet}{RGB}{148,0,211}
\definecolor{twelve}{RGB}{100,100,170}
\definecolor{thirteen}{RGB}{100,150,50}
\definecolor{fourteen}{RGB}{200,0,0}
\definecolor{fifteen}{RGB}{0,200,0}
\definecolor{sixteen}{RGB}{0,0,200}
\definecolor{seventeen}{RGB}{200,0,200}
\definecolor{eighteen}{RGB}{0,200,200}
\usepackage{tikz-cd}

\usepackage{centernot}
\usepackage{longtable}
\usepackage{mathtools}
\usepackage{stmaryrd}
\makeatletter
\newcommand{\xMapsto}[2][]{\ext@arrow 0599{\Mapstofill@}{#1}{#2}}
\def\Mapstofill@{\arrowfill@{\Mapstochar\Relbar}\Relbar\Rightarrow}
\makeatother

\newcommand{\hana}[1]{{\color{violet}{#1}}}

\usepackage[all]{xy}
\newtheorem{thm}{Theorem}[section]
\newtheorem*{theorem*}{Theorem}
\newtheorem*{conjecture*}{Conjecture}
\newtheorem*{corollary*}{Corollary}

\newtheorem{cor}[thm]{Corollary}

\newtheorem{prop}[thm]{Proposition}

\theoremstyle{definition}
\newtheorem{defin}[thm]{Definition}
\newtheorem{exm}[thm]{Example}
\newtheorem{rem2}[thm]{Remark}
\newtheorem*{rem2*}{Remark}

\newtheorem{thmx}{Theorem}

\def\a{\mathbb{A}}

\def\c{\mathbb{C}}

\def\f{\mathbb{F}}

\def\q{\mathbb{Q}}
\def\r{\mathbb{R}}
\def\s{\mathbb{S}}

\def\z{\mathbb{Z}}

\def\cofib{\operatorname{cofib}}
\def\fib{\operatorname{fib}}
\def\ker{\operatorname{ker}}
\def\coker{\operatorname{coker}}

\def\denom{\operatorname{denom}}

\usepackage{amssymb}
\usepackage{tikz-cd}

\definecolor{DarkBlue}{rgb}{.1, 0.35, 0.6} 
\definecolor{DarkBrown}{rgb}{.5, 0.2, 0.2} 
\hypersetup{
    colorlinks,
    linkcolor={DarkBrown},
    citecolor={DarkBlue},
    urlcolor={blue!80!black}
}

\author{Hana Jia Kong}\address{Harvard University}\email{hana.jia.kong@gmail.com}
\author{J.D. Quigley}\address{University of Virginia}\email{mbp6pj@virginia.edu}
\title[The slice spectral sequence for $L$ over prime fields]{The slice spectral sequence for a motivic analogue of the connective $K(1)$-local sphere}

\begin{document}
\maketitle

\begin{abstract}
We compute the slice spectral sequence for the motivic stable homotopy groups of $L$, a motivic analogue of the connective $K(1)$-local sphere over prime fields of characteristic not two. Together with the analogous computation over algebraically closed fields, this yields information about the motivic $K(1)$-local sphere over arbitrary base fields of characteristic not two. To compute the slice spectral sequence, we prove several results which may be of independent interest.  We describe the $d_1$-differentials in the slice spectral sequence in terms of the motivic Steenrod operations over general base fields, building on analogous results of Ananyevskiy, R{\"o}ndigs, and {\O}stv{\ae}r for the very effective cover of Hermitian K-theory. We also explicitly describe the coefficients of certain motivic Eilenberg--MacLane spectra and compute the slice spectral sequence for the very effective cover of Hermitian K-theory over prime fields. 

\end{abstract}

\tableofcontents

\section{Introduction}

Motivic stable homotopy theory, or the stable homotopy theory of algebraic varieties, was developed by Morel and Voevodsky in \cite{MV99} to apply powerful techniques from stable homotopy theory to problems in algebraic geometry and number theory.  Since motivic stable homotopy groups are the universal stable $\a^1$-invariant, a great deal of effort has gone into understanding the \emph{motivic stable stems}, i.e., the motivic stable homotopy groups of the unit in the motivic stable homotopy category. This paper studies certain patterns in an approximation to the motivic stable stems over prime fields of characteristic not two. Using analogous results over algebraically closed fields, this establishes the existence of similar patterns over arbitrary base fields of characteristic not two. 

To situate our results, we briefly survey the existing analyses of the motivic stable stems. Over arbitrary base fields, the low-dimensional Milnor--Witt stems\footnote{The $n$-th Milnor--Witt stem is the sum over $i \in \z$ of $\pi_{n+i,i}^F(\s)$, where $\s$ is the unit. See \Cref{SS:Conventions} for our indexing conventions on motivic stable homotopy groups.} are connected to Milnor--Witt K-theory \cite{Mor12} and Milnor K-theory, Hermitian K-theory, and motivic cohomology \cite{RSO19, RSO21}. Over certain base fields, various completions of the motivic stable stems have been computed in larger ranges \cite{BCQ21, BI22, DI10, DI16, Isa19, IWX20, OO14, Wil16, WO17}. 

The most well-studied pattern in the motivic stable stems is $\eta$-periodicity, where $\eta$ is the first motivic Hopf map. The $\eta$-periodic motivic stable stems have been computed over various base fields \cite{AM17, CQ21, GI15, GI16, OR20, Wil18}. Over general base fields and Dedekind domains, the $\eta$-periodic sphere spectrum sits in a fiber sequence with connective Witt theory \cite{BH20, Bac22}. 

The pattern we study in this paper, $v_1$-periodicity, is much less well-understood. The $v_1$-periodic motivic stable stems have been computed over the algebraically closed fields of characteristic zero \cite{CQ21}, and a small number of $v_1$-periodic families have been produced over general base fields \cite{Qui21c, Qui21b}. In \cite{BIK22}, the coefficients of a motivic spectrum $L$ which captures the $v_1$-periodic phenomena in the motivic stable stems (cf. \Cref{Rmk:JQvsL} and \Cref{Rmk:Ltov1}) were computed over the complex numbers and the real numbers. 

Our main result (\Cref{MT:L}) is a computation of the coefficients of $L$ over all prime fields of characteristic not two and all algebraically closed fields. Since any field sits between a prime field and an algebraically closed field in a sequence of field extensions, our computations identify the piece of the coefficients of $L$ common across every base field of characteristic not two. 

\subsection{Summary of results}

From now on, we work in the $2$-complete setting and only work over base fields $F$ of characteristic not two. In \cite{BH20}, Bachmann and Hopkins defined Adams operations on Hermitian K-theory and its very effective cover \cite{ARO17}
$$\psi^3: KQ \to KQ, \quad \psi^3 : kq \to kq.$$
Recently, Balderrama, Ormsby, and the second author \cite{BOQ23} have proven that
\begin{equation}\label{Eqn:FibL}
L_{KGL/2}\s \simeq \fib(\psi^3-1: KQ \to KQ),
\end{equation}
where $KGL/2$ is mod two algebraic K-theory and $L_{(-)}$ denotes Bousfield localization. This is a motivic analogue of the classical identification of the $K(1)$-local sphere,
$$L_{K(1)}\s \simeq \fib(\psi^3-1: KO \to KO).$$
Following \cite{BIK22}, we define a motivic spectrum $L$ by
$$L := \fib(\psi^3-1: kq \to kq).$$ 
The motivic spectrum $L$ is related to a motivic analogue of the connective $K(1)$-local sphere as follows:

\begin{rem2}\label{Rmk:JQvsL}
We can identify the very effective cover functor $\tilde f_0$ with $f_0\tau_{\geq 0},$ the effective cover $f_0$ of the connective cover $\tau_{\geq 0}$ with respect to the homotopy $t$-structure (e.g. \cite[Rem.2]{ARO17}). The functor $f_0$ is triangulated but $\tau_{\geq 0}$ is not, so the very effective cover functor $\tilde f_0$ does not preserve fiber sequences and
$L$ is not equivalent to the very effective cover of $\tilde{f}_0 L_{KGL/2} \s.$ However, the two are closely related. 

Consider the fiber sequences 
$L_{KGL/2} \s \to KQ \to KQ$ and $F \to \tau_{\geq 0}KQ \to \tau_{\geq 0}KQ$. Since $f_0$ is triangulated, we have $f_0 F = L.$ There is a map $\tau_{\geq 0} L_{KGL/2} \to F$ whose cofiber $D$ has homotopy groups concentrated in homotopy $t$-structure degree $-1$: 
$$\pi_{*-1, *}^F (D)= \coker(\pi_{*,*}^F KQ\xrightarrow{\psi^3-1} \pi_{*,*}^F KQ).$$ 
Passing to effective covers then gives a comparison between $\tilde f_0 L_{KGL/2}$ and $L$. 
\end{rem2}

In this paper, we compute the coefficients of $L$ using the slice spectral sequence. Our main result is the following:

\begin{thmx}\label{MT:L}
The groups $\pi_{**}^F(L)$ are described in \Cref{Sec:L} for the following base fields $F$:
\begin{enumerate}
\item For $F = \bar{F}$ algebraically closed in \Cref{SS:CL}. 
\item For $F = \f_q$ a finite field of characteristic not two in \Cref{SS:FqL}.
\item For $F = \q_q$ in \Cref{SS:QqL} ($q$ odd) and \Cref{SS:Q2L} ($q=2$). 
\item For $F = \r$ in \Cref{SS:RL}. 
\item For $F = \q$ in \Cref{SS:QL}. 
\end{enumerate}
\end{thmx}

\begin{rem2}
The groups $\pi_{**}^F(L)$ were already computed for $F = \c$ and $F=\r$ by Belmont, Isaksen, and the first author in \cite{BIK22}, so our new contribution is the computation over prime fields and the $q$-adic rationals. 
\end{rem2}

One simple consequence of our computations is that a famous pattern of $2$-torsion in the classical stable stems sits in $\pi_{**}^F(L)$ for all fields $F$. 
\begin{cor}\label{Cor:Bernoulli}
Let $F$ be any field of characteristic not two. Then $\pi_{4k-1,2k}^F(L)$ contains a summand of order at least the $2$-component of $\denom \left( \frac{B_{2k}}{4k} \right)$, where $B_{2k}$ is the $2k$-th Bernoulli number. 
\end{cor}

\begin{proof}
Any field $F$ sits in a sequence of field extensions
$$k \to F \to \bar{F}$$
between a prime field $k$ and its algebraic closure $\bar{F}$. Inspecting the results of \Cref{Sec:L}, we see that the corollary holds for $k$ and $\bar{F}$, and moreover, the relevant summands in $k$ base change to the corresponding summands in $\bar{F}$. It follows by naturality that they must be nonzero in $F$. 
\end{proof}

\begin{rem2}\label{Rem:imJ}
More generally, the groups $\pi_*(\tau_{\geq 0} L_{K(1)}\s)$ appear as subgroups of $\pi_{**}^F(L)$ in certain bidegrees 
for all primes fields and algebraically closed fields. Therefore the classical $v_1$-periodic stable stems sit interestingly inside $\pi_{**}^F(L)$ for all fields of characteristic not two. In particular, we find the $v_1$-periodic elements constructed in \cite{Qui21c, Qui21b}. 
\end{rem2}

\begin{rem2}\label{Rmk:Ltov1}
In classical stable homotopy theory, the groups $\pi_*(\tau_{\geq 0} L_{K(1)}\s)$ encodes the image of the $J$-homomorphism and the Hurewicz image of connective real K-theory, i.e., all of the $v_1$-periodic elements in the classical stable stems. Over algebraically closed fields of characteristic zero, the groups $\pi_{**}^{\bar{F}}(L)$ computed in \cite{BIK22} capture the $v_1$-periodic elements in the $\bar{F}$-motivic stable stems (which were computed in \cite{CQ21}). The extension of this observation to more general base fields is the subject of ongoing investigation and is closely related to the motivic analogue of the Telescope Conjecture.
\end{rem2}

\begin{rem2}
Classically, the map $\psi^3-1: ko \to ko$ is closely related to $d_1$-differentials in the $ko$-based Adams spectral sequence \cite{Mah81, DM89}. The $kq$-based motivic Adams spectral sequence was studied in \cite{CQ21}, where it was shown that the $n$-th Milnor--Witt stem is detected in filtration at most $n$. We hope to apply our analysis of $\psi^3-1: kq \to kq$ in future work to study low-dimensional Milnor--Witt stems via the $kq$-based motivic Adams spectral sequence, potentially extending the results of Morel \cite{Mor12} and R{\"o}ndigs--Spitzweck--{\O}stv{\ae}r \cite{RSO19, RSO21} mentioned above to higher Milnor--Witt stems. 
\end{rem2}

\Cref{Cor:Bernoulli} and \Cref{Rem:imJ} highlight some features of $v_1$-periodic motivic stable homotopy theory which are independent of the base field. We include the following remark to illustrate an aspect of the computation which does depend on the base field.

\begin{rem2}\label{Rmk:FqDifferentials}
As we will explain below, we compute $\pi_{**}^{F}(L)$ using the effective slice spectral sequence. When $F = \c$, there is an important family of $d_1$-differentials
$$d_1(\iota v_1^2 \tau^n) = \iota \tau^{n+1} h_1^3$$
which occur for all $n \geq 0$. However, when $F = \f_q$, the similar differential $d_1(\iota x_q v_1^2 \tau^n)$ occurs if and only if 
$$
\begin{cases}
\nu(q-1)+\nu(n+1) > 3 \quad & \text{ if } q \equiv 1 \mod 4, \\
\text{unconditionally} \quad & \text{ if } q \equiv 3 \mod 4 \text{ and } n \equiv 0 \mod 2, \\
\nu(q^2-1)+\nu(n+1) > 4 \quad & \text{ if } q \equiv 3 \mod 4 \text{ and } n \equiv 1 \mod 2,
\end{cases}
$$
where $\nu(-)$ denotes dyadic valuation. Further discussion appears in \Cref{SS:FqL}. 
\end{rem2}

We compute $\pi_{**}^F(L)$ over prime fields using the effective slice spectral sequence (ESSS), a powerful tool which has been used to great effect in recent years (e.g., \cite{KRO20, RO16, RSO19, RSO21}). The $E_1$-term of the ESSS for $L$ is comprised of shifted copies of $\pi_{**}^F(H\z/2^n)$ for varying $n$.
The coefficient rings of $H\z/2$ and $H\z$ have been recorded over prime fields \cite{Hil11, Orm11, OO13, Kyl15} (see \Cref{Sec:Background} for a summary), but as far as we are aware, the coefficient rings of $H\z/2^n$ have not appeared in previous literature. 

\begin{thmx}\label{MT:Z2n}
The groups $\pi_{**}^F(H\z/2^n)$ are described for all $n \geq 1$ and all $F \in \{ \bar{F}, \f_q, \q_q, \r, \q \}$ in \Cref{Sec:Z2n}. 
\end{thmx}

Even with these groups in hand, describing the $E_1$-term of the ESSS for $L$ in a manner suitable for computations is a difficult problem. One novel aspect of our work is the concise graphical calculus developed in \Cref{Sec:kq} and \Cref{Sec:L}. Previous computations over $\f_q$ (e.g., \cite{Kyl15, WO17}) or $\q_q$ (e.g., \cite{OO13}) often use different charts to display information for different congruence classes of $q$, while computations over $\r$ (e.g., \cite{BI22, DI16a, BIK22}) usually use different charts for different Milnor--Witt stems. In our work, however, we are able to display complicated data, like the $E_1$-page of the ESSS for $L$ over $\f_q$, $q$ odd, in a single chart (\Cref{Fig:FqE1L}).\footnote{Roughly speaking, the tradeoff is that instead of using multiple charts containing a small variety of symbols, we use a single chart containing a large variety of symbols.}

Once we have described the $E_1$-term of the ESSS for $L$, we are tasked with computing differentials. In all of the cases we consider, the differentials are forced by comparison with the ESSS for $L$ over algebraically closed fields, the ESSS for $kq$ over the given prime field, or in the case $F=\q$, comparison with the ESSS for $L$ over $\r$ and $\q_q$. Using the defining fiber sequence defining $L$, we are also able to deduce formulas for the $d_1$ differentials over arbitrary base fields in terms of motivic Steenrod operations (see Thm. \ref{thm:sliceL}). 

\begin{thmx}\label{MT:kq}
The effective slice spectral sequence for $kq$ is explicitly described in \Cref{Sec:kq} for the following base fields:
\begin{enumerate}
\item For $F = \bar{F}$ algebraically closed in \Cref{SS:Ckq}. 
\item For $F = \f_q$ a finite field in \Cref{SS:Fqkq}.
\item For $F = \q_q$ in \Cref{SS:Qqkq} ($q$ odd) and \Cref{SS:Q2kq} ($q=2$). 
\item For $F = \r$ in \Cref{SS:Rkq}. 
\item For $F = \q$ in \Cref{SS:Qkq}. 
\end{enumerate}
\end{thmx}

\begin{rem2}
The effective slice spectral sequences for $kq$ over $F = \bar{F}$ and $F = \r$ were already described in \cite{BIK22}. In fact, \cite{ARO17}, Ananyevskiy, R{\"o}ndigs and {\O}stv{\ae}r expressed the differentials in the ESSS for $kq$ in terms of Steenrod operations over any field of characteristic not two. \Cref{MT:kq} provides the more explicit analysis in terms of generators and relations needed to analyze the ESSS for $L$. 
\end{rem2}

While describing \emph{how} the differentials in the ESSS for $L$ are produced is easy, actually determining \emph{which} differentials occur and describing the pattern intelligibly is difficult. As we alluded to in \Cref{Rmk:FqDifferentials}, the characteristic of the base field determines whether or not certain $d_1$-differentials occur. Our approach is to find arithmetic conditions under which differentials occur, and then describe the groups in the next page of the ESSS in terms of these conditions. We refer the reader to \Cref{SS:FqL} for examples and further discussion. 

\begin{rem2}
One consequence of our analysis is that the Hasse map
$$\pi_{**}^{\q}(E) \to \prod_{\nu} \pi_{**}^{\q_\nu}(E),$$
where $\nu$ ranges over all places, is injective for $E = kq$ and $E = L$. In the terminology of Ormsby--{\O}stv{\ae}r \cite{OO13}, $kq$ and $L$ satisfy the motivic Hasse principle. In \cite{BCQ21}, Balderrama, Culver, and the second author showed that the $E_2$-term of the motivic Adams spectral sequence converging to the motivic stable stems satisfies the motivic Hasse principle. Since $kq$ and $L$ serve as first approximations to the motivic sphere spectrum, it is interesting to wonder if the motivic sphere spectrum itself might satisfy the motivic Hasse principle. 
\end{rem2}

\begin{rem2}
In \cite[Rmk. 5.14]{OO13}, Ormsby--{\O}stv{\ae}r explain how their computation of $\pi_{**}^{\q}BPGL\langle 1 \rangle$ and \cite[Lem. 2.9]{OO13} can be used to recover the Rognes--Weibel computation of the $2$-complete algebraic K-theory of $\q$ from \cite{RW00}. Similarly, our computations can be used to recover the $2$-complete Hermitian K-theory of $\bar{F}$, $\f_q$, $\q_q$, $\r$, and $\q$ using \cite[Lem. 1.5.12]{Kyl15} which relates $\pi_{**}kq$ to $\pi_{**}KQ$. For example, using \Cref{MT:kq}, we can recover some of the Hermitian K-theory computations of Friedlander \cite{Fri76} and Berrick--Karoubi \cite{BK05}. 
\end{rem2}

\begin{rem2}
Since $L$ is the fiber of the map $\psi^3-1: kq \to kq$, it is natural to ask if computing $\pi_{**}^FL$ via the resulting resulting long exact sequence in homotopy groups would be simpler than using the ESSS. We mention three reasons the ESSS is preferable:

\begin{enumerate}

\item To analyze the coefficients of $L$ over rings of integers in number fields (cf. \Cref{Rmk:OF}), it is no longer possible to work with generators and relations, so computations via long exact sequence become intractable. Our computations over $\q$ using the ESSS in this paper could serve as a guide for these more elaborate computations. 

\item Although the coefficients of $kq$ have been explicitly described over $\c$ \cite{IS11} and $\r$ \cite{GHIR19}, and the ESSS for $kq$ has been studied over general base fields \cite{ARO17, KRO20, RO16}, the coefficients of $kq$ had not been computed explicitly enough over prime fields to make $\pi_{**}^FL$ easily accessible via long exact sequence. Since we already need to use the ESSS to compute $\pi_{**}^Fkq$, the ESSS is a natural tool for computing $\pi_{**}^FL$. 

\item As explained on \cite[Pg. 2]{BIK22}, the effective slice filtration is part of the ``higher structure" of $\pi_{**}^FL$ which may be helpful for future analysis. 

\end{enumerate}

\end{rem2}

\begin{rem2}\label{Rmk:OF}
In a previous version of this paper, we outlined an approach to analyzing the coefficients of $L$ over rings of $S$-integers in number fields, where $S$ contains the archimedean and dyadic places. We were able to give a complete description of the $d_1$-differentials in the slice spectral sequence in terms of Steenrod operations (analogous to computations in \cite{RSO19, RSO21, KRO20}), but ran into two problems. First, the description of the $E_2$-term in terms of subquotients of mod two motivic cohomology was quite complicated. Second, even with the $E_2$-term computed, there was still room for longer differentials, the sources and targets of which were subquotients of mod two motivic cohomology groups, but not mod two motivic cohomology groups themselves. In particular, the longer differentials cannot be expressed in terms of Steenrod operations. It could still be possible to describe these differentials by relating them to analogous differentials over $\r$ (cf. \cite{KRO20}), but we did not end up pursuing this. 
\end{rem2}

\subsection{Outline}

In \Cref{Sec:Background}, we recall the spectrum $L$ and its essential properties. We also recall the coefficients of $H\z/2$ and $H\z$ over our fields of interest. 

In \Cref{Sec:Z2n}, we compute the coefficient rings of $H\z/2^n$ for all $n \geq 1$ over the relevant fields. Our approach using the Adams spectral sequence may be of independent interest. 

In \Cref{Sec:kq}, we recall the ESSS for $kq$ over algebraically closed fields and the reals from \cite{BIK22}, and explicitly describe the ESSS for $kq$ over prime fields and the $q$-adic rationals. The results in this section are used in \Cref{Sec:L} to produce some differentials in the next section. 

In \Cref{Sec:L}, we recall the ESSS for $L$ over algebraically closed fields and the reals from \cite{BIK22}, and analyze the ESSS for $L$ over prime fields and the $q$-adic rationals. 

\Cref{Sec:Figures} includes charts for many of the spectral sequences we consider and \Cref{Sec:Tables} includes tables describing the coefficients of $L$ over certain fields. 

\subsection{Conventions}\label{SS:Conventions}

\begin{enumerate}

\item We implicitly work in the $2$-complete setting. 

\item We work over fields of characteristic not two. In particular, $\f_q$ always refers to a finite field of odd characteristic. 

\item Motivic stable homotopy groups are bigraded in the form $(s,w)$, where $s$ denotes the stem and $w$ denotes the motivic weight. 

\item We use the abbreviation `ESSS' for the effective slice spectral sequence. 

\item We write $s_*(X)$ for the slices of a motivic spectrum $X$. 

\item The horizontal axis in ESSS charts is always the stem $s$ and the vertical axis is always the ``Adams--Novikov filtration" $f$, which is twice the effective slice filtration minus the stem. 

\item $E_r^F(X)$ denotes the $E_r$-page of the ESSS for an $F$-motivic spectrum $X$. If the field is understood, we will sometimes suppress it from the notation. 

\item We take the elements $\mathsf h$ and $\rho$ to be as defined in the conventions of \cite{BIK22}. 
\end{enumerate}

\subsection{Acknowledgments}

The authors thank Oliver R{\"o}ndigs for discussions related to \cite{KRO20} and a reminder about the convergence of the ESSS, as well as William Balderrama and Kyle Ormsby for discussions relating this work to the $KGL/2$-local sphere. The first author was supported by National Science Foundation grant DMS-1926686. The second author is grateful to the Max Planck Institute for Mathematics in Bonn for its hospitality and financial support, and was partially supported by NSF grants DMS-2039316, DMS-2314082, and an AMS-Simons Travel Grant.

\section{Background}\label{Sec:Background}

In this section, we provide the background necessary for our computations. The motivic spectrum $L$ and its essential properties are recalled in \Cref{SS:LSlices}, and the coefficients of the Eilenberg--MacLane spectra $H\z/2$ and $H\z$ over our fields of interest are recalled in \Cref{SS:Coeffs}.

\subsection{The motivic spectrum $L$ and its slices}\label{SS:LSlices}

In \cite{BH20}, Bachmann and Hopkins constructed a unital ring map
$$\psi^3: kq\left[\frac{1}{3}\right] \to kq\left[\frac{1}{3}\right]$$
whose Betti realization is the classical Adams operation $\psi^3$. We define
$$L:= \fib\left(\psi^3-1: kq\left[\frac{1}{3}\right] \to kq\left[\frac{1}{3}\right]\right).$$

We will use the effective slice spectral sequence (ESSS) \cite{Lev08, RSO19}
$$E_1^{s,f,w}(X) = \pi_{s,w}(s_{\frac{s+f}{2}}(X)) \Rightarrow \pi_{s,w}(X)$$
to study $\pi_{**}^FL$ for various base fields $F$. The differentials in the ESSS have the form 
$$d_r: E_r^{s,f,w} \to E_r^{s-1,f+2r-1,w},$$
cf. \cite[Thm. 2.7]{BIK22}. 

The $E_1$-page of the ESSS for $L$ can be computed using the fiber sequence of slices
$$s_*L \to s_*kq \xrightarrow{\psi^3-1} s_* kq,$$
which implies
$$\pi_{**}s_*L \cong K \oplus \Sigma^{-1} C$$
where
$$K := \ker(\pi_{**}(\psi^3-1)) \quad \text{ and } \quad C := \coker(\pi_{**}(\psi^3-1)).$$
As in \cite{BIK22}, we will write $\iota x$ for the image of $x \in \Sigma^{-1}C$ under the inclusion into $\pi_{**}s_*L$. 

We use the grading convention $(s,f_{AN},w)$ as in \cite{BIK22}. Here $s$ is the stem, $w$ is the weight, and $f_{AN}$ is the ``Adams--Novikov filtration'', which equals twice the stem minus the slice filtration $2f-s$.

By \cite[Thm. 17]{ARO17}, the slices of $kq$ can be expressed by the formula
\begin{equation}
\label{eq:slicekq}
	s_*kq = H\z[h_1,v_1^2]/(2h_1),
\end{equation}
where $|v_1^2| = (4,0,2)$ and $|h_1| = (1,1,1)$. 
We explain the meaning of this formula. In $H\z[h_1,v_1^2]/(2h_1)$, a monomial of tri-degree $(s,f_{AN},w)$ contributes a summand of the motivic Eilenberg--MacLane spectrum $\Sigma^{s,w}HA$ to the $w$th slice $s_wkq$.

\begin{rem2}
\label{rem:jandL}
	In \cite{BH20}, Bachmann and Hopkins showed that the map $\psi^3-1$ factors  
	\begin{equation}
	\label{eq:kqtoksp}
kq\xrightarrow{\psi^3-1} \Sigma^{4,2} ksp	
\end{equation}
	where the target is the very effective cover of $\Sigma^{4,2}kq$. 
	After localizing at $2$, the fiber of \eqref{eq:kqtoksp} defines a motivic spectrum $j$ that is analogous to the classical image-of-$J$ spectrum.
	
	It can be shown that the slices of $j$ can be expressed by
	$$s_*j = H\z \otimes \{\alpha\text{-family classes}\}$$
	where ``$\alpha\text{-family}$ classes'' refers to the 
classical Adams--Novikov classes (at the prime $2$) in filtration $1$ and their $\alpha_1$-power multiples.

The slices of $L$ and $j$ differs by a tower of suspensions of $H\f_2$, and in slice filtrations $-1$ and $2$. 
More precisely, the very effective cover functor gives $\Sigma^{4,2}ksp\to kq$, which induces a map $j\to L.$ The difference between $s_*L$ and $s_*j$ is given by 
$$s_*\cofib(j\to L)= H\z\otimes \Big(\z[h_1]/(2h_1)\{\iota\} \oplus \f_2\{\bar\alpha_{2} \} \Big)$$
where $|\iota|=(-1,1,0)$ and $|\bar \alpha_{2}|=(3,1,2).$

One can compute the coefficients of $j$ over prime fields using the techniques in this manuscript, but the computation is slightly messier in low degrees.
\end{rem2}

\begin{thm}
\label{thm:convergence}
Over the base fields $F = \bar{F}, \f_q, \q_q, \r, \q$,
the slice spectral sequence for $L$ and the slice spectral sequence for $kq$ converge strongly to $\pi_{**}^F(L_2^\wedge)$ and $\pi_{**}^F(kq_2^\wedge)$.
\end{thm}

\begin{proof}
	The case for $\r$ is proved in \cite{BIK22}. 
	By \cite{BIK22}, the limit of the $2$-completed slice tower of $L$ is $(2,\eta)$-completed $L$. Over the specified base fields, the $(2,\eta)$-completions of $L$ and $kq$ coincide with the $2$-completions of $L$ and $kq$ by \cite{HKO11a} (in characteristic zero) and \cite[Prop. 5.10]{WO17} (in positive characteristic). The result then follows 
	from \cite[Thm.~7.1]{Boa98} and the calculations in \S 4 for $kq$ and \S 5 for $L$.
\end{proof}

\subsection{Coefficient rings of $H\mathbb{Z}/2$ and $H\mathbb{Z}$}\label{SS:Coeffs}

For future reference, we record the coefficient rings of $H\z/2$ and $H\z$ over various base fields in this section. We refer the reader to \cite[Sec. 2.1]{IO18} for a summary of how $\pi_{**}^{F}(H\z/2)$ can be computed using Milnor K-theory \cite{Mil69} and the Bloch--Kato Conjecture \cite{Voe03, Voe11}. All of the computations of $\pi_{**}^{F}(H\z/2)$ appear in \cite{IO18}, except the case $F = \q$ which we pull from \cite{OO13}. Once $\pi_{**}^F(H\z/2)$ is known, the groups $\pi_{**}^F(H\z)$ can be computed using the $\rho$-Bockstein spectral sequence \cite{Hil11} and motivic Adams spectral sequence \cite{DI10}, or by consulting previous motivic cohomology computations. 

\subsubsection{Algebraically closed fields}

These results follow, for instance, from \cite{DI10}. Let $F = \bar{F}$ be an algebraically closed field. Then we have
$$\pi_{**}^{F}(H\z/2) \cong \z/2[\tau], \quad \pi_{**}^{F}(H\z) \cong \z_2[\tau],$$
where $|\tau| = (0,-1)$.

\subsubsection{Finite fields}

These results appear, for instance, in \cite{Kyl15}. We have
$$\pi_{**}^{\f_q}(H\z/2) \cong 
\begin{cases} 
\z/2[\tau,u]/(u^2) \quad & \text{ if } q \equiv 1 \mod 4, \\ 
\z/2[\tau,\rho]/(\rho^2) \quad & \text{ if } q \equiv 3 \mod 4, 
\end{cases}
$$
and
$$\pi_{**}^{\f_q}(H\z) \cong 
\begin{cases} 
\z\{1\} \oplus \bigoplus_{i \geq 0} \z/2^{\nu(q-1)+\nu(i+1)} \{u\tau^i\} \quad & \text{ if } q \equiv 1 \mod 4, \\
\z\{1\} \oplus \bigoplus_{i \geq 0, \ even} \z/2\{\rho \tau^i\} \\
\oplus \bigoplus_{i \geq 0, \ odd} \z/2^{\nu(q^2-1)+\nu(i+1)-1} \{\rho \tau^i\} \quad & \text{ if } q \equiv 3 \mod 4,
\end{cases}
$$
where $|\tau| = (0,-1)$ and $|u| = |\rho| = (-1,-1)$. 

In order to make our analysis of finite fields independent of the congruence class of $q$, we make the following definition. 

\begin{defin}\label{Def:FqExponents}
For all $i \geq 0$ and odd prime powers $q$, we define
$$s_q(i) := \begin{cases}
\nu(q-1)+\nu(i+1), \quad & \text{ if } q \equiv 1 \mod 4, \\
1 \quad & \text{ if } q \equiv 3 \mod 4 \text{ and } i \equiv 0 \mod 2, \\
\nu(q^2-1) + \nu(i+1) - 1 \quad & \text{ if } q \equiv 3 \mod 4 \text{ and } i \equiv 1 \mod 2.
\end{cases}
$$
\end{defin}

With this notation, we have for all odd prime powers $q$ that 
$$\pi_{**}^{\f_q}(H\z) \cong \z\{1\} \oplus \bigoplus_{i \geq 0} \z/2^{s_q(i)}\{x_q\tau^i\},$$
where $x_q=u$ if $q \equiv 1 \mod 4$ and $x_q=\rho$ if $q \equiv 3 \mod 4$. 

\subsubsection{$\q_q$ with $q$ odd}

These results appear in \cite{Orm11}. We have
$$\pi_{**}^{\q_q}(H\z/2) \cong 
\begin{cases}
\z/2[\tau,\pi,u]/(\pi^2, u^2) \quad & \text{ if } q \equiv 1 \mod 4, \\
\z/2[\tau, \pi, \rho]/(\rho^2, \rho \pi + \pi^2) \quad & \text{ if } q \equiv 3 \mod 4
\end{cases}
$$
where $|\tau| = (0,-1)$, $|\pi| = |u| = |\rho| = (-1,-1)$. The homotopy groups of $H\z$ over $\q_q$, $q$ odd, appear in \cite[Thm. 5.8]{Orm11}:
$$\pi_{**}^{\q_q}(H\z) \cong 
\begin{cases}
\z\{1\} \quad & \text{ if } (*,*) = (0,0), \\
\z\{\pi\} \oplus \z/2^{s_q(0)}\{x_q\} \quad & \text{ if } (*,*) = (-1,-1), \\
\z/2^{s_q(0)}\{x_q\pi\} \quad & \text{ if } (*,*) = (-2,-2), \\
\z/2^{s_q(i-1)}\{\pi^{\epsilon-1}x_q \tau^{i-1}\} \quad & \text{ if } (*,*) = (-\epsilon, -i+1-\epsilon), \ i \geq 2, \ \epsilon =1 \text{ or } 2,
\end{cases}
$$
where $x_q = u$ if $q \equiv 1 \mod 4$ and $x_q=\rho$ if $q \equiv 3 \mod 4$. Equivalently, there is an additive isomorphism
\begin{equation}
\pi_{**}^{\q_q}(H\z) \cong \left(\z\{1\} \oplus \bigoplus_{i \geq 0} \z/2^{s_q(i)} \{x_q \tau^{i}\} \right) \otimes \z[\pi]/(\pi^2) \cong \pi_{**}^{\f_q}(H\z) \otimes \z[\pi]/(\pi^2);
\label{Eqn:QqToFq}
\end{equation}
if $q \equiv 1 \mod 4$, this is an isomorphism of graded rings. 

\subsubsection{$\q_2$}

These results appear in \cite{OO13}. We have
$$\pi_{**}^{\q_2}(H\z/2) \cong \z/2[\tau,\pi, u, \rho]/ (\rho^3, u^2, \pi^2, \rho u, \rho \pi, \rho^2 + u\pi),$$
where $|\tau| = (0,-1)$, $|\pi| = |u| = |\rho| = (-1,-1)$. The coefficients of $H\z$ over $\q_2$ appear in the proof of \cite[Thm. 3.19]{OO13} as
$$
\pi_{**}^{\q_2}(H\z) \cong 
\begin{cases}
\z\{1\} \quad & \text{ if } (*,*) = (0,0), \\
\z\{u\} \oplus \z\{\pi\} \oplus \z/2\{\rho\} \quad & \text{ if } (*,*) = (-1,-1), \\
\z\{y_m\tau^m\} \oplus \z/2^{s_{3}(m)}\{z_m \tau^m\} \quad & \text{ if } (*,*) = (-1,-m-1),\ m>0 , \\
\z/2^{s_3(m)}\{\rho^2 \tau^m\} \quad & \text{ if } (*,*) = (-2,-m-2),\ m\geq 0 , \\
0 \quad & \text{ otherwise.}
\end{cases}
$$
where
$y_m$ is $u$ for $m$ odd and $\pi $ for $m$ even, and $z_m$ is $\pi$ for $m$ odd and $\rho$ for $m $ even.

\subsubsection{$\r$}

These results appear, for instance, in \cite{Hil11}. We have
$$\pi_{**}^{\r}(H\z/2) \cong \z/2[\tau,\rho], \quad \pi_{**}^{\r}(H\z) \cong \z[\tau^2,\rho]/(2\rho).$$

\subsubsection{$\q$}

These results appear in \cite[Sec. 5]{OO13}. The groups $\pi_{**}^{\q}(H\z/2)$ can be obtained from the mod two Milnor K-theory of $\q$ \cite[Prop. 5.3]{OO13},
$$
k_n^M(\q) \cong 
\begin{cases}
\z/2\{1\} \quad & \text{ if } n=0, \\
\z/2\{\rho\} \oplus \bigoplus_{p \geq 2, \ prime} \z/2\{[p]\}, \quad & \text{ if } n=1, \\
\z/2\{\rho^2 \} \oplus \bigoplus_{p \geq 3, \ prime} \z/2\{a_p\}, \quad & \text{ if } n=2, \\
\z/2\{\rho^n\} \quad & \text{ if } n \geq 3,
\end{cases}
$$
by tensoring with $\z/2[\tau]$. The multiplicative structure and $\rho$-module structure is described further in \cite[Props. 5.3-5.4]{OO13}. The Hasse map
$$k^M_*(\q) \to \prod_v k_*^M(\q_v)$$
sends pure symbols to their obvious images in $\prod k_1^M(\q_v)$, $a_p$ to the unique nonzero class in $k_2^M(\q_p)$ and to $0$ in $k_2^M(\q_\ell)$, $\ell \neq p$. 

\begin{rem2}\label{Rmk:HZ2Alternate}
To simplify our comparison with $\pi_{**}^{\q}(H\z)$ and $\pi_{**}^{\q_\nu}(H\z/2)$ in the sequel, we note that 
$$\pi_{**}^{\q}(H\z/2) \cong \pi_{**}^{\r}(H\z/2) \oplus \z/2[\tau]\{[2]\} \oplus \bigoplus_{q \text{ odd}} \z/2[\tau]\{[q], a_q\}.$$
The summand $\pi_{**}^{\r}(H\z/2)$ maps isomorphically onto $\pi_{**}^{\r}(H\z/2)$ under the map
$$\pi_{**}^{\q}(H\z/2) \to \pi_{**}^{\r}(H\z/2),$$
and for each prime $q$, the elements $[q]$ and $a_q$ map to $\pi$ and $\pi x_q$, respectively, under
$$\pi_{**}^{\q}(H\z/2) \to \pi_{**}^{\q_q}(H\z/2).$$
\end{rem2}

The groups $\pi_{**}^{\q}(H\z)$ are more complicated. Specializing \cite[Thm. 5.13]{OO13} to the case $n=0$, we have
$$\pi_{**}^{\q}(H\z) \cong A \oplus B \oplus C,$$
where $A = \bigoplus_{p \equiv 3 \mod 4, \ prime} A_p$ with $A_p$ the bigraded abelian group defined by
$$(A_p)_{**} = 
\begin{cases}
\z \quad & \text{ if } (*,*) = (-1,-1), \\
\z/2 \quad & \text{ if } (*,*) = (-2,-2r-2), \ r \geq 0, \\
\z/2^{s_p(2r+1)} \quad & \text{ if } (*,*) = (-2,-2r-3), \ r \geq 0, \\
0 \quad & \text{ otherwise,}
\end{cases}
$$
$B = \bigoplus_{p \equiv 1 \mod 4, \ prime} B_p$, with
$$(B_p)_{**} = 
\begin{cases}
\z \quad & \text{ if } (*,*) = (-1,-1), \\
\z/2^{s_p(r)} \quad & \text{ if } (*,*) = (-2, -2-r), \ r \geq 0, \\
0 \quad & \text{ otherwise,}
\end{cases}
$$
and $C = C'(0) \oplus C'''(0)$, with
$$C'(0) = \z_2[\rho,\tau^2]/(2\rho),$$
$$C'''(0)_{**} = 
\begin{cases}
\z \quad & \text{ if } (*,*) = (-1,-2r-1), \ r \geq 0, \\
\z/2^{3+\nu(r+1)} \quad & \text{ if } (*,*) = (-1, -2r-2), \ r \geq 0, \\
0 \quad & \text{ otherwise.}
\end{cases}
$$

\begin{rem2}\label{Rmk:HZAlternate}
The groups $\pi_{**}^{\q}(H\z)$ have an alternative description which will be useful in the sequel. Let
$$D_q := \begin{cases}
A_q \quad & \text{ if } q \equiv 3 \mod 4, \\
B_q \quad & \text{ if } q \equiv 1 \mod 4,
\end{cases}
$$
so
$$\pi_{**}^{\q}(H\z) \cong C \oplus \bigoplus_{q \text{ odd}} D_q .$$
We have
$$(D_q)_{**} \cong \begin{cases}
\pi_{*+1,*+1}^{\f_q}(H\z) \quad & \text{ if } (*,*)\neq (0,0), \\
0 \quad & \text{ if } (*,*)=(0,0). 
\end{cases}
$$ 
Here, $\pi_{*+1,*+1}^{\f_q}(H\z)$ is shorthand for the $\pi$-divisible part of $\pi_{**}^{\q_q}(H\z)$, cf. \eqref{Eqn:QqToFq}. 
Moreover, we observe that the summands $C'(0)$ and $C'''(0)$ of $C$ can be identified with familiar objects:
$$C'(0) \cong \pi_{**}^{\r}(H\z),$$
and there is an obvious inclusion
$$C'''(0) \hookrightarrow \pi_{**}^{\q_2}(H\z)$$
with image those subgroups generated by classes of the form $\pi \tau^m$, $m \geq 0$. 

\cite[Sec. 5]{OO13} implies that all of these identifications are realized via the maps
$$\pi_{**}^{\q}(H\z) \to \pi_{**}^{\q_\nu}(H\z),$$
where $\q_\nu$ ranges over all places.
\end{rem2}

\section{Coefficient rings of $H\z/2^n$}\label{Sec:Z2n}

In \Cref{Sec:L}, we will see that the slices of $L$ are comprised of suspensions of the Eilenberg--MacLane spectra $H\z/2^n$ for various $1 \leq n \leq \infty$. In this section, we record the coefficients of these spectra over our fields of interest.

The motivic Adams spectral sequence \cite{DI10} for $H\z$ has the following form.
$$E_1^{s,w,t}= \pi_{t-s,w}(H\z/2)[h_0] \implies  \pi_{t-s,w}H\z.$$
Here we use the Adams grading; the element $h_0$ has degree $(s,w,t)=(1,0,1).$ The abutment is known, and this information determines the differentials (see \cite{Kyl15} for $F = \f_q$ and \cite{OO13} for $F = \q_q, \q$).

Similarly, we have the motivic Adams spectral sequence for $H\z/2^n$:
\begin{equation}
\label{eq:2BSS}
	E_1^{s,w,t}= \pi_{t-s,w}(H\z/2)[h_0]/h_0^{n}\implies  \pi_{t-s,w}(H\z/2^n).
\end{equation}
The differentials can be recovered from the motivic Adams spectral sequence for $H\z$. 

\begin{rem2}
Over algebraically closed fields and over the real numbers, we can compute $\pi_{**}^F(H\z/2^n)$ directly using the long exact sequence in homotopy associated to the cofiber sequence
$$H\z \xrightarrow{\cdot 2^n} H\z \to H\z/2^n.$$ 
We have
$$\pi_{**}^{\bar{F}}(H\z/2^n) \cong \z/2^n[\tau] \text{\quad and \quad}
\pi_{**}^{\r}(H\z/2^n) \cong \z/2^n[\tau^2,\rho]/(2\rho).$$
\end{rem2}

\begin{prop}
\label{prop:HZ}
	The differentials in the $F$-motivic Adams spectral sequence for $H\z/2^n$ are determined via the Leibniz rule by the following ($i\geq 1$):
	\begin{enumerate}
	\item When $F = \f_q$:
	$$d_{s_q(i-1)}\tau^{i}=x_q \tau^{i-1}h_0^{s_q(i-1)}, ~s_q(i-1)< n,$$
	where $x_q = u$ if $q \equiv 1 \mod 4$ and $x_q=\rho$ if $q \equiv 3 \mod 4$.
	\item When $F = \q_q$:
	$$d_{s_q(i-1)}\tau^{i}=x_q \tau^{i-1}h_0^{s_q(i-1)}, ~s_q(i-1)< n,$$
	where $x_q = u$ if $q \equiv 1 \mod 4$ and $x_q=\rho$ if $q \equiv 3 \mod 4$.
		\item When $F = \q_2$:
		$$d_1\tau=\rho h_0, \quad d_{3+\nu(i)}\tau^{2i}=\pi \tau^{2i-1}h_0^{3+\nu(i)}, \  \nu(i)< n-3.$$
		\item When $F = \q$:
		$$d_1 \tau = \rho h_0, \quad d_1[p]\tau = (\rho^2 + a_p)h_0 \quad p \equiv 3 \mod 4, $$
		$$d_{s_p(i-1)}[p]\tau^{i} = a_p \tau^{i-1} h_0^{s_p(i-1)}, \quad 1 \leq s_p(i-1) < n \text{~and~} i \text{~even for}~ p \equiv 3 \mod 4, $$
		$$d_{3+\nu(i)}\tau^{2i}h_0 = [2]\tau^{2i-1} h_0^{4+\nu(i)}, \quad \text{ with } \nu(i)<n-3.$$
	\end{enumerate}

\end{prop}

\begin{proof}
	\begin{enumerate}
		\item This follows from \cite[Lem. 4.2.1, Lem. 4.2.2]{Kyl15}.\footnote{Note that for $q\equiv 3 $ mod $4$, the exponent of $h_0$ in the target is off by one in \emph{loc. cit.}}
		\item This follows from \cite[Thm. 3.16, Thm. 3.17]{OO13}.
		\item This follows from \cite[Thm. 3.18, Thm. 3.19]{OO13}.
		\item This follows from \cite[Thm. 5.5, Thm. 5.8]{OO13}. 
	\end{enumerate}
\end{proof}

\begin{prop}
\label{prop:HZ-2tor}
The coefficient rings of $H\z/2^n$, $1\leq n < \infty$, are given by the following formulas. In cases (1) and (2), we write $i(w)$ for $\text{min}(s_q(w-1),n).$
\begin{enumerate}
	\item When $F = \f_q$ with $q\equiv 1 $ mod $4$:
		$$\pi_{s,w}^{\f_q} H\z/2^n=
		\begin{cases}
		\z/2^{n} \{1\} \quad & \text{ if } s = 0, w= 0,\\
\z/2^{i(-w)} \{2^{n-i(-w)}\tau^{-w}\} \quad & \text{ if } s = 0, w\leq -1, \\
\z/2^{i(-w)}\{\tau^{-1-w}x_q\} \quad & \text{ if } s=-1, w \leq -1,  \\
0 \quad & \text{ otherwise.}
\end{cases}$$
Here $x_q = u$ if $q \equiv 1 \mod 4$ and $x_q=\rho$ if $q \equiv 3 \mod 4$.
		\item When $F = \q_q$ with $q\equiv 1 $ mod $4$:
		$$\pi_{s,w}^{\q_q} H\z/2^n=
		\begin{cases}
\z/2^{i(-w)} \{2^{n-i(-w)}\tau^{-w}\} \quad & \text{ if } s = 0, w\leq 0, \\
\z/2^{i(-w)}\{\tau^{-1-w}x_q\}\oplus
\z/2^{i(-w-1)}\{2^{n-i(-w-1)}\tau^{-1-w}\pi\} \quad & \text{ if } s=-1, w \leq -1  \\
\z/2^{i(-w-1)}\{\tau^{-2-w}\pi x_q\} \quad & \text{ if } s=-2, w \leq -2\\
0 \quad & \text{ otherwise,}
\end{cases}$$
\item When $F=\q_2$:
$$
\pi_{s,w}^{\q_2}(H\z/2^n) \cong 
\begin{cases}
\z/2^{i(-w)} \{2^{n-i(-w)}\tau^{-w}\} \quad & \text{ if } s = 0, w\leq 0,\\
\z/2^n\{\tau^{-1-w}u\}\oplus \z/2^{i(-w)}\{2^{n-i(-w)}\tau^{-1-w}\pi\} \\
\oplus \z/2\{2^{n-1}\tau^{-1-w}\rho\} \quad & \text{ if } s=-1, w \leq -2 \text{ even}, \\
\z/2^{i(-1-w)}\{2^{n-i(-1-w)}\tau^{-1-w}u\}\oplus \z/2^n\{\tau^{-1-w}\pi\} \\
\oplus \z/2\{\tau^{-1-w}\rho\} \quad & \text{ if } s=-1, w \leq -1 \text{ odd}, \\
\z/2^{i(-w-1)} \{2^{n-i(-w-1)}\tau^{-2-w}\rho^2\} \quad & \text{ if } s= -2, w\leq -1, \\
0 \quad & \text{ otherwise,}
\end{cases}
$$
where $i(w)=1$ when $w$ is odd and $\text{min}(2+\nu(w),n)$ when $w$ is even.
\item When $F = \q$: 
$$\pi_{**}^{\q}(H\z/2^n) \cong C{'''}_{**}(0) \oplus C{'}_{**}(0) \oplus \bigoplus_{q \text{ odd}} D_q(n),$$
where
$$(D_q(n))_{**} \cong \begin{cases}
\pi_{*+1,*+1}^{\f_q}(H\z/2^n) \quad & \text{ if } (*,*)\neq (0,0), \\
0 \quad & \text{ if } (*,*)=(0,0),
\end{cases} 
$$ 
$$C'_{**}(0) \cong \pi_{**}^{\r}(H\z/2^n),$$
and
$$C'''(0) \hookrightarrow \pi_{**}^{\q_2}(H\z/2^n)$$ 
can be identified with the subgroups generated by classes of the form $\pi \tau^m$, $m \geq 0$. 
Here, $\pi_{*+1,*+1}^{\f_q}$ is shorthand for the $\pi$-divisible part of $\pi_{**}^{\q_q}$. 
\end{enumerate}
	
\end{prop}

\begin{proof}
The results follow from the motivic Adams spectral sequence in \eqref{eq:2BSS} and differentials in \Cref{prop:HZ}.	
\end{proof}

\section{Coefficient rings of $kq$}\label{Sec:kq}

In this section, we describe the effective slice spectral sequence for $kq$ over various base fields.

\begin{thm}[Thm. 3.2, \cite{ARO17}]
\label{thm:kqslice}
The nonnegative slices of $kq$ are as follows:
\[
s_{2q} kq \simeq \bigvee_{0 \leq i < q} \Sigma^{2q+2i,2q} H\z/2  \vee \Sigma^{4q,2q} H\z , 
\]
\[
s_{2q+1} kq \simeq \bigvee_{0 \leq i \leq q} \Sigma^{2q+1+2i,2q} H\z/2.
\]
The negative slices of $kq$ are zero. 
\end{thm}

For the $d_1$-differentials, note that $d_1(q): s_q kq \to \Sigma^{1,0} s_{q+1} kq$ is a map between finite sums of suspensions of motivic Eilenberg--MacLane spectra for $\z/2$ and $\z$, so it can be described via its restrictions $d_1(q,i)$ to $\Sigma^{q+i,q} HA$. Since $[H\z/2^n, \Sigma^{s,1}H\z/2^m] = 0$ for $s \geq 5$, the differential $d_1(q,i)$ splits into at most three nontrivial components.

\begin{thm}[Thm. 3.5, \cite{ARO17}]
The $d_1$-differential in the slice spectral sequence for $kq$ is given by\\
$$d_1 kq(q, i) = \begin{cases}
(0, Sq^2, Sq^3Sq^1), ~~q - 1 > i \equiv 0 \mod 4, \\
(\tau,Sq^2 +\rho Sq^1,Sq^3Sq^1), ~~ q-1>i\equiv 2\mod4,
 \end{cases}$$
$$d_1kq(q, q -1) = \begin{cases}
 	(0, Sq^2, \partial Sq^2Sq^1) ,~~q \equiv 1\mod 4, \\
 	(\tau,Sq^2+\rho Sq^1,\partial Sq^2Sq^1),~~ q\equiv 3\mod4,
 \end{cases}$$
 $$d_1kq(q,q)= \begin{cases}
 	(0,Sq^2 \circ pr,0),~~ q\equiv 0\mod4,\\
 	(\tau \circ pr,Sq^2\circ pr,\partial Sq^2Sq^1),~~ q\equiv 2\mod4.
 \end{cases}$$

\end{thm}

\subsection{Algebraically closed fields}\label{SS:Ckq}

Let $F = \bar{F}$ be an algebraically closed field. The effective slice spectral sequence for $kq$ over $F$ was computed in \cite[Sec. 3]{BIK22}. The $E_1$-term of the ESSS is depicted in \Cref{Fig:CE1kq}. 

The differential $d_1(v_1^2) = \tau h_1^3$ follows by Betti realization \cite[Prop. 3.2]{BIK22}, with all remaining differentials following by the Leibniz rule. The spectral sequence collapses at $E_2$; the resulting $E_\infty$-page is depicted in \Cref{Fig:CEookq}. 

\subsection{Finite fields}\label{SS:Fqkq}

The $E_1$-page of the ESSS over $\f_q$ can be obtained from the $E_1$-page of the ESSS over algebraically closed fields (\Cref{Fig:CE1kq}) by letting a rectangle denote $\z$ instead of $\z[\tau]$, and then juxtaposing a copy of the $E_1$-page, shifted by $(-1,1)$, with each rectangle replaced by a diamond denoting the $u$- or $\rho$-divisible part $\pi_{**}^{\f_q}(H\z)$. The resulting $E_1$-page is depicted in \Cref{Fig:FqE1kq}. 

The differential $d_1(v_1^2) = \tau h_1^3$ follows from base change to the algebraic closure, and all remaining differentials follow by the Leibniz rule. The resulting pattern of $d_1$-differentials is depicted in \Cref{Fig:FqE1kq}.

The ESSS collapses for degree reasons at $E_2$. The $E_\infty$-page is shown in \Cref{Fig:FqEookq}. The hidden extensions $\mathsf h \cdot 2v_1^{2+4k} u = \tau^2 h_1^3v_1^{4k}$, $k \geq 0$, follow by comparison to Kylling's computation of $\pi_{**}^{\f_q}(kq)$ via the motivic Adams spectral sequence \cite[Sec. 4]{Kyl15}.

\subsection{$\mathbb{Q}_q$ with $q$ odd}\label{SS:Qqkq}

The $E_1$-page of the ESSS over $\q_q$ can be obtained from the $E_1$-page of the ESSS over $\f_q$ using the additive isomorphism \eqref{Eqn:QqToFq}, which implies that there is an additive isomorphism 
$$E_1^{\q_q}(kq) \cong E_1^{\f_q}(kq) \otimes \z[\pi]/(\pi^2),$$
where $|\pi| = (-1,0,-1)$. Graphically, this means the $E_1$-term over $\q_q$ consists of two copies of the $E_1$-page over $\f_q$: one copy as it appears in \Cref{Fig:FqE1kq}, and one copy shifted by $(-1,1)$. The result is depicted in \Cref{Fig:QqE1kq}.

As in the case of finite fields, all $d_1$-differentials follow from base change to the algebraic closure and $\pi^\delta x_q^\epsilon \tau^n$-linearity for the appropriate choices of $\delta, \epsilon \in \{0,1\}$. The resulting $E_\infty$-page is identical the the $E_\infty$-page over $\f_q$, tensored with $\z[\pi]/(\pi^2)$. Hidden extensions follow from similar arguments. The resulting groups are depicted in \Cref{Fig:QqEookq}.

\subsection{$\mathbb{Q}_2$}\label{SS:Q2kq}

Since $\rho^2$ is non-zero over $\q_2$, the ESSS over $\q_2$ has a different $d_1$-differential pattern. The $d_1$-differentials in the ESSS for $KQ$ (and consequently, $kq$) were computed in terms of Steenrod operations in \cite[Thm. 5.5]{RO16}. The $E_1$-term is depicted in \Cref{Fig:Q2E1kq}.

In \Cref{Fig:Q2E1kq}, the red differentials are given by taking $Sq^2$, the brown differentials are given by 
multiplying by $\tau$, and the blue differentials are $Sq^2+\rho Sq^1$

\begin{exm}\label{exm:Q2kq}
We explain the $d_1$ calculation using sample computations. We suggest readers compare the following examples with Figure \ref{Fig:Q2E1kq} and Figure \ref{Fig:Q2E2kq}. 
\begin{enumerate}
	\item There is a red differential between the class $\z/2[\tau]\{h_1\}$ at $(1,1)$ and the class $\z/2[\tau]\{h_1^2 \rho^2\}$ at $(0,4)$. The formula for the differential is 
$$d_1(\tau^{4n+2}h_1)=\tau^{4n+1}\rho^2 h_1^2 
\text{~and~}
d_1(\tau^{4n+3} h_1)=\tau^{4n+2} \rho^2  h_1^2, n\geq 0.
$$
Therefore, on the $E_2$-page, we have a $\z/2\{1,\tau^3\}[\tau^4]$ in the degree of the target and a $\z/2\{1,\tau\}[\tau^4]$ in the degree of the source.
\item There is a differential from the class $\pi_{-2,*}^{\q_2}(H\z)\{v_1^2\}$ at $(2,2)$, and another one from the class $\z/2[\tau]\{h_1^2\}$ at the same degree, both hitting $\z/2[\tau]\{\rho^2 h_1^3\}$
at $(1,5)$. The formulas are 
$$d_1(\tau^{n}\rho^2 v_1^2 )=\tau^{n+1}\rho^2 h_1^3, \text{~and~} d_1(\tau^{i}h_1^2 )=\tau^{i-1}\rho^2 h_1^3, ~~i=4k+2, 4k+3. 
$$
We see that 
$\tau^{4n+1}\rho^2v_1^2$ and $\tau^{4n+3}h_1^2$ hit the same target, and $\tau^{4n}v_1^2$ and $\tau^{4n+2}h_1^2$ also hit the same target. Therefore, the sums of the following forms survive:
 $$\tau^{4n+2}h_1^2 + \tau^{4n}\rho^2v_1^2
 \quad  \text{~and~} \quad
 \tau^{4n+3}h_1^2 + \tau^{4n+1}\rho^2v_1^2, \quad n\geq 0.
 $$
 As a result, the surviving classes at degree $(2,2)$ are those of the form above (represented by the green diamond), and 
 $\tau^{4n}, \tau^{4n+1}$ for $n\geq 0$ (represented by the black dot with subscript 1).


At degree $(3,3)$, the surviving classes are 
$$\tau^{4n} h_1^3, \tau^{4n+5} h_1^3, \tau^{4n+3} h_1^3 + \tau^{4n+1}\rho^2 v_1^2, \tau^{4n+2} h_1^3 + \tau^{4n}\rho^2 v_1^2 \text{~for~} k\geq 0,$$ represented by a violet dot. 
The $h_1$-extensions are suggested by the names.
\end{enumerate}
	
\end{exm}

The $E_2$-page is depicted in \Cref{Fig:Q2E2kq}. There are no room for $d_2$-differentials, therefore we have $E_2=E_\infty.$

\subsection{$\mathbb{R}$}\label{SS:Rkq}

The ESSS for $kq$ over $\r$ is the subject of \cite[Sec. 4]{BIK22}. There are two interesting differentials over $\r$ \cite[Prop. 5.2]{BIK22}:
$$d_1(v_1^2) = \tau h_1^3, \quad d_1(\tau^2) = \rho^2 \tau h_1.$$
The $E_1$-page is depicted in \cite[Fig. 5]{BIK22}, and the $E_\infty$-page is depicted in \cite[Figs. 6-8]{BIK22}. 

\subsection{$\mathbb{Q}$}\label{SS:Qkq}

We will now analyze the ESSS for $kq$ over $\q$. To begin, recall the presentations of $\pi_{**}^{\q}(H\z/2)$ and $\pi_{**}^{\q}(H\z)$ from \Cref{Rmk:HZ2Alternate} and \Cref{Rmk:HZAlternate}. Since $E_1^{\q}(kq)$ is comprised of shifts of these groups, it naturally decomposes as
$$E_1^{\q}(kq) \cong E_1^+(kq)  \oplus E_1^-(kq),$$
where $E_1^+(kq) \subseteq E_1^{\q}(kq)$ is obtained by making the replacements
$$\pi_{**}^{\q}(H\z/2) \rightsquigarrow \bigoplus_{q \text{ odd}} \z/2[\tau]\{[q],a_q\}, \quad \pi_{**}^{\q}(H\z) \rightsquigarrow \bigoplus_{q \text{ odd}} D_q,$$
and $E_1^-(kq) \subseteq E_1^{\q}(kq)$ is obtained by making the replacements
$$\pi_{**}^{\q}(H\z/2) \rightsquigarrow \pi_{**}^{\r}(H\z/2) \oplus \z/2[\tau]\{[2]\}, \quad \pi_{**}^{\q}(H\z) \rightsquigarrow C \cong \pi_{**}^{\r}(H\z) \oplus C'''(0).$$

Graphically, we can obtain $E_1^{\q}(kq)$ from $E_1^{\bar{F}}(kq)$ (\Cref{Fig:CE1kq}) by replacing each bullet by a copy of $\pi_{**}^{\q}(H\z/2)$ and replacing each square by a copy of $\pi_{**}^{\q}(H\z/2)$; the replacements are shown in \Cref{Fig:QF2Z}. 

The purpose of this decomposition is to understand the map
$$E_1^{\q}(kq) \to E_1^{\q_\nu}(kq)$$
for each place $\nu$ of $\q$. In particular, we see that the Hasse map
$$E_1^{\q}(kq) \to \prod_{\nu} E_1^{\q_\nu}(kq)$$
is injective. It follows that the nontrivial $d_1$-differentials over $\q$ are generated by 
$$d_1([q] \tau^i v_1^{2k} ) = [q] \tau^{i+1} h_1^3 v_1^{2k-2}, \quad d_1(a_{q'} \tau^i v_1^{2k}) = a_{q'} \tau^{i+1} h_1^3 v_1^{2k-2}, \quad d_1(\tau^2) = \rho^2 \tau h_1,$$
where $q$ ranges over all primes, $q'$ ranges over all odd primes, and $i$ and $k$ range over all nonnegative integers. Note that none of the more exotic differentials (red ones in \Cref{Fig:Q2E1kq}) from the $\q_2$-case lift to $\q$ since the sources and targets do not lift to $\q$. 

There is no room for longer differentials, so $E_2^{\q}(kq) \cong E_\infty^{\q}(kq)$. All hidden extensions can be handled by comparison with the local places. 

\begin{rem2}
Our description of the $E_1$-term over $\q$ makes clear that the map
$$E_1^{\q}(kq) \to \prod_{\nu} E_1^{\q_\nu}(kq)$$
is injective. Since all of the differentials over $\q$ are lifted from differentials over the local places, we conclude that the Hasse map
$$\pi_{**}^{\q}(kq) \to \prod_{\nu} \pi_{**}^{\q_\nu}(kq)$$
is injective, i.e., $kq$ satisfies the motivic Hasse principle in the sense of \cite[Sec. 4]{OO13}, 
\end{rem2}

\section{Coefficient rings of $L$ over prime fields}\label{Sec:L}

In this section, we describe the effective slice spectral sequence for $L$ over various base fields.


\begin{prop}
We have
$$s_0 L \simeq \Sigma^{-1,0} H\z \vee H\z,$$
and the positive slices of $L$ are as follows:
$$
s_q L\simeq \bigvee_{-1\leq i \leq q-2}\Sigma^{q+i, q} H\z/2 \vee \Sigma^{2q-1,q}H\z/2^{a_q},
$$
where $a_q=\nu_2(3^q-1).$ The negative slices of $L$ are zero. 
\end{prop}

\begin{proof}
This follows from the long exact sequence in homotopy associated to the fiber sequence of motivic spectra
$$s_q L \to s_q kq \xrightarrow{\psi^3-1} s_q kq.$$
\end{proof}

We use the following convention to denote the slices of $L$.
Let $L(q,i)$ denote the summand $\Sigma^{q+i-1,q}H\z/2^n$ in the $q$th slice, where $n$ is $a_q$ if $q=i>0$, $\infty$ if $q-0$, and $1$ otherwise.

Consider maps between Eilenberg--MacLane spectra. We have 
$
[H\z/2^n, \Sigma^{s,1}H\z/2^m]= 0
$
 for $s\geq 5.$ Therefore, restricting to the summand $L(q,i)=\Sigma^{q+i-1,q}H\z/2^n$ in the $q$th slice, the summands in the $(q+1)$th slice that has possibly nontrivial $d_1$ components are $L(q,j)$ for $ i-1\leq j\leq i+3$.

\begin{thm}
\label{thm:sliceL}
When $char(F) \neq  2$ the $d_1$ differential in the slice spectral sequence for $L$ is given by\\
$$d_1 L(q, i) = \begin{cases}
(Sq^3Sq^1, 0, Sq^2, 0, 0), ~~q - 1 > i \equiv 0 ~~mod~~ 4, \\
(Sq^3Sq^1,0, Sq^2,0,0 ), ~~ q-1>i\equiv 1~~ mod~~4,\\
(Sq^3Sq^1,0, Sq^2 +\rho Sq^1,0,\tau ), ~~ q-1>i\equiv 2~~ mod~~4,\\
(Sq^3Sq^1,0, Sq^2 +\rho Sq^1,0, \tau ), ~~ q-1>i\equiv 3~~ mod~~4,\\
 \end{cases}$$
 
 $$d_1 L(q, q) = \begin{cases}
(0,Sq^2\partial_2, Sq^2 \circ pr_2,0,0),~~ q\equiv 0~~mod~~4, \\
(0,inc^2\circ Sq^2Sq^1, Sq^2,0,0),~~ q\equiv 1~~mod~~4,\\
(0,Sq^2\partial_2,Sq^2\circ pr_2, \tau \circ \partial_2, \tau \circ pr_2),~~ q\equiv 2~~mod~~4,\\
(0,inc^2\circ Sq^2Sq^1,Sq^2+\rho Sq^1,0,\tau),~~ q\equiv 3~~mod~~4,
 \end{cases}$$

  $$d_1 L(q, q-1) = \begin{cases}
(Sq^3Sq^1,0, Sq^2+\rho Sq^1,0,\tau),~~ q\equiv 0~~mod~~4, \\
(\partial^2 Sq^2 Sq^1,0,  Sq^2, 0,0),~~ q\equiv 1~~mod~~4,\\
(Sq^3Sq^1,0 , Sq^2,0,0),~~ q\equiv 2~~mod~~4,\\
(\partial^2 Sq^2Sq^1,0, Sq^2+\rho Sq^1, 0, \tau ),~~ q\equiv 3~~mod~~4.
 \end{cases}$$

Here $\partial_2, \partial^2, inc^2, pr_2$ denote the maps in the homotopy cofiber sequence
$$H\z/2^{n}\to H\z/2^{n+1}\xrightarrow{pr_2} H\z/2\xrightarrow{\partial^2} \Sigma^{1,0}H\z/2^n,$$
$$H\z/2\xrightarrow{inc^2} H\z/2^{n+1}\xrightarrow{} H\z/2^{n}\xrightarrow{\partial_2} \Sigma^{1,0}H\z/2$$ 
for some $n$, where $n$ is decided by the target slice.

\end{thm}

\begin{proof}

The results are deduced from \Cref{thm:kqslice} and the homotopy cofiber sequence 
\begin{equation}
\label{eq:defofLseq}
	L\to kq\xrightarrow{\psi^3-1}kq
\end{equation}
that defines $L$. We explain by way of example how to deduce the formulas. 

Consider the differential on $L(q,q)$ for $q\equiv 0~mod~4.$ By (\ref{eq:defofLseq}), this slice summand is obtained by the homotopy cofiber sequence 
$$\Sigma^{-1,0}kq (q,q)\xrightarrow{p} L(q,q)\xrightarrow{d} kq(q,q) \xrightarrow{\times a_q} kq (q,q).$$
By \Cref{thm:kqslice}, both $kq(q,q)$ support a differential expressed by $(0,Sq^2 \circ pr,0)$.
Thus we have
$$
\begin{tikzcd}
\Sigma^{-(2q-1,q)}\Sigma^{-1,0}kq(q,q)\ar[equal,r]&H\z \ar[r,"pr"]\ar[d,"p"'] & H\z/2 \ar[r,"Sq^2"] & H\z/2\\
\Sigma^{-(2q-1,q)}L(q,q)\ar[equal,r]&H\z/2^n\ar[ur, "pr^{2^{a_q}}_2"']\ar[dr, "\partial^{2^{a_q}}_2"]
\ar[d,"d"'] & &\\ 	
\Sigma^{-(2q-1,q)}kq(q,q)\ar[equal,r] &\Sigma H\z \ar[r,"pr"] &  \Sigma H\z/2 \ar[r,"Sq^2"] & \Sigma H\z/2
\end{tikzcd}.
$$
The composite from the middle to the top right gives the third component $Sq^2\circ pr_2$; the composite from the middle to the bottom right gives the second component $Sq^2\circ \partial_2$. Therefore we obtain that the $d_1$ differential on $L(q,q)$ is 
$(0,Sq^2\partial_2, Sq^2 \circ pr_2,0,0)$.
\end{proof}

\subsection{Algebraically closed fields}\label{SS:CL}

The ESSS for $L$ over algebraically closed fields was computed in \cite[Sec. 3]{BIK22}. The multiplicative generators of the $E_1$-term appear in \cite[Table 3]{BIK22} and the $E_1$-term is shown in \cite[Fig. 3]{BIK22}. The $d_1$-differentials appear in \cite[Table 4]{BIK22} and the resulting $E_2=E_\infty$-page appears in \cite[Fig. 4]{BIK22}. Hidden extensions are handled in \cite[Prop. 3.16]{BIK22}. 

For ease of reference, we have reproduced \cite[Fig. 3]{BIK22} and \cite[Fig. 4]{BIK22} in \Cref{Fig:CE1L} and \Cref{Fig:CEooL}, respectively. Additive generators are given in \Cref{Table:CLGens}, with relations appearing immediately after. 

\subsection{Finite fields}\label{SS:FqL}

To compute the $E_1$-term of the ESSS for $L$ over finite fields, we must analyze the effect of $\psi^3-1$ on the slices of $kq$. Comparison with the algebraically closed case implies that the only possible nontrivial action of $\psi^3-1$ is on the integer slices, i.e.,
$$\psi^3 -1: \pi_{**}^{\f_q}(H\z)\{v_1^{2k}\} \to \pi_{**}^{\f_q}(H\z)\{v_1^{2k}\},$$
which may be identified with 
$$\cdot 2^{\nu(3^{2k}-1)} = \cdot 2^{\nu(k)+3} : \pi_{**}^{\f_q}(H\z)\{v_1^{2k}\} \to \pi_{**}^{\f_q}(H\z)\{v_1^{2k}\}.$$
Recalling the functions $s_q(i)$ and notation $x_q$ from \Cref{Def:FqExponents}, we may write this more explictily as
$$\cdot 2^{\nu(k)+3} : \z\{v_1^{2k}\} \oplus \bigoplus_{i \geq 0} \z/2^{s_q(i)}\{v_1^{2k}x_q \tau^i\}.$$
The kernel and cokernel of
$$\cdot 2^{\nu(k)+3}: \z\{v_1^{2k}\} \to \z\{v_1^{2k}\}$$
are $0$ and $\z/2^{\nu(k)+3}$, respectively. On the components indexed by $x_q v_1^{2k} \tau^i$, 
$$\cdot 2^{\nu(k)+3} : \z/2^{s_q(i)}\{x_q v_1^{2k} \tau^i\} \to \z/2^{s_q(i)}\{x_q v_1^{2k} \tau^i\},$$
we have
$$\ker \cong \begin{cases}
\z/2^{s_q(i)}\{x_q v_1^{2k}\tau^i\} \quad & \text{ if } s_q(i) \leq \nu(k)+3, \\
\z/2^{\nu(k)+3}\{2^{s_q(i)-\nu(k)-3} x_q v_1^{2k}\tau^i\} \quad & \text{ if } s_q(i) > \nu(k)+3,
\end{cases}
$$
and
$$\coker \cong \begin{cases}
\z/2^{s_q(i)}\{x_q v_1^{2k}\tau^i\} \quad & \text{ if } s_q(i) \leq \nu(k)+3, \\
\z/2^{s_q(i)-\nu(k)-3}\{x_q v_1^{2k} \tau^i\} \quad & \text{ if } s_q(i) > \nu(k)+3.
\end{cases}
$$

\begin{defin}\label{Def:KCFq}
For each $k \geq 0$, we define graded groups $K_q(k)$ and $C_q(k)$ as follows. Let 
$$K_q(0) = C_q(0) := \pi_{-1,*}^{\f_q}(H\z)$$
denote the $x_q$-divisible part of $\pi_{**}^{\f_q}(H\z)$. For all $k \geq 1$, let
$$K_q(k) := \bigoplus_{i \geq 0} \ker(\cdot 2^{\nu(k)+3} : \z/2^{s_q(i)}\{x_q v_1^{2k} \tau^i\} \to \z/2^{s_q(i)}\{x_q v_1^{2k}\tau^i\}),$$
$$C_q(k) := \bigoplus_{i \geq 0} \coker(\cdot 2^{\nu(k)+3} : \z/2^{s_q(i)}\{x_q v_1^{2k} \tau^i\} \to \z/2^{s_q(i)}\{x_q v_1^{2k}\tau^i\})$$
be the sums of the kernels and cokernels, respectively, described above. 

Let $\tilde{K}_q(k)$ denote the groups obtained from $K_q(k)$ as follows. If $k$ is even, then $\tilde{K}_q(k) := K_q(k)$. If $k$ is odd, then:
\begin{enumerate}
\item If $s_q(i) \leq \nu(k)+3$, then the $i$-th summand of $\tilde{K}_q(k)$ is the group obtained from the $i$-th summand of $K_q(k)$ by decreasing the order of $2$-torsion by exactly $1$ in each summand. 
\item If $s_q(i) > \nu(k)+3$, the $i$-th summand of $\tilde{K}_q(k)$ is the same as the $i$-th summand in $K_q(k)$. 
\end{enumerate}

Let $\tilde{C}_q(k)$ denote the groups obtained from $C_q(k)$ as follows. If $k$ is even, then $\tilde{C}_q(k) := C_q(k)$. If $k$ is odd, then $\tilde{C}_q(k)$ is obtained by decreasing the order of $2$-torsion by exactly $1$ in each summand of $C_q(k)$. 
\end{defin}

\Cref{Def:KCFq} allows us to concisely describe the $E_1$- and $E_\infty$-terms of the ESSS for $L$ in chart form. The $E_1$-term is depicted in \Cref{Fig:FqE1L}.

All of the differentials follow from base change to the algebraic closure and $x_q$-linearity of the $d_1$-differentials. Since $s_q(i) \geq 1$ for all $i \geq 0$, the effect of each \emph{nontrivial} differential is easy to describe: simply reduce the order of $2$-torsion in the source and target by $2$. However, describing which differentials are nontrivial is somewhat subtle.

The red differentials out of the dark green boxes (i.e., out of $\z/2^?\{\iota \tau^i v_1^{2k}\}$) and dark green left-pointing triangles (i.e., out of $C_q(k)$) are nontrivial for all $i \geq 0$, but the green differentials out of the right-pointing triangles (i.e., out of $K_q(k)$) are trivial on some summands. Recall that the generator of the $i$-th summand in $K_q(k)$ for $k$ odd is 
$$
\begin{cases}
x_q v_1^{2k} \tau^i \quad & \text{ if } s_q(i) \leq 3, \\
2^{s_q(i)-\nu(k)-3} x_q v_1^{2k} \tau^i \quad & \text{ if } s_q(i) > 3.
\end{cases}
$$
Here, we have used that $\nu(k)=0$ if $k$ is odd to rewrite $\nu(k)+3$ as $3$. 

If $s_q(i) \leq 3$, then there is a differential
$$d_1(x_q v_1^{2k} \tau^i) = x_q h_1^3 v_1^{2k-2} \tau^{i+1},$$
but if $s_q(i) > \nu(k)+3$, then we must have
$$d_1(2^{s_q(i)-\nu(k)-3}x_q v_1^{2k} \tau^i) = 0$$
since
$$d_1(2^{s_q(i)-\nu(k)-3}v_1^{2k})=0$$
by base change to the algebraic closure. 

\begin{rem2}
In practice, it is straightforward to determine on which summands the green differentials are nontrivial for a given $q$. For instance, if $q=3$, then:
\begin{itemize}
\item If $i \not\equiv 3 \mod 4$, then $s_3(i) \leq 3$. Indeed, if $i$ is even, then $s_3(i) = 1 \leq 3$, and if $i \equiv 1 \mod 4$, then $s_3(i) = \nu(3^2-1) + \nu(i+1) - 1 = 2+\nu(i+1) = 3 \leq 3$. 
\item If $i \equiv 3 \mod 4$, then $s_3(i) = \nu(3^2-1)+\nu(i+1)-1 = 2+\nu(i+1) > 3$.
\end{itemize}
\end{rem2}

This leads us to the $E_2$-term. There is no room for higher differentials, so $E_2=E_\infty$. All hidden extensions follow from comparison to the algebraic closure or $\tau^n x_q$-linearity. The resulting groups are depicted in \Cref{Fig:FqEooL}, and multiplicative generators and relations are listed in \cref{Table:FqLGens}.

\subsection{$\mathbb{Q}_q$ with $q$ odd}\label{SS:QqL}

As with $kq$ (\Cref{SS:Qqkq}), the computation of $\pi_{**}^{\q_q}(L)$ for $q$ odd follows from the analogous computation over $\f_q$ (\Cref{SS:FqL}). The additive isomorphism \eqref{Eqn:QqToFq} implies that
$$E_1^{\q_q}(L) \cong E_1^{\f_q}(L) \otimes \z[\pi]/(\pi^2).$$
Thus the $E_1$-term over $\q_q$ can be obtained from \Cref{Fig:FqE1L} by superimposing an identical copy shifted by $(-1,1)$. Moreover, the $d_1$-differentials occur in both copies; this follows as usual from base change to the algebraic closure and $\pi^\delta x_q^\epsilon \tau^n$-linearity for the appropriate choices of $n \geq 0$ and $\delta, \epsilon \in \{0,1\}$. 

There is no room for further differentials, so $E_2 = E_\infty$. By the above discussion, the $E_\infty$-term satisfies
$$E_\infty^{\q_q}(L) \cong E_\infty^{\f_q}(L) \otimes \z[\pi]/(\pi^2),$$
so it can be obtained from \Cref{Fig:FqEooL} by superimposing the same picture shifted by $(-1,1)$. 

\subsection{$\mathbb{Q}_2$}\label{SS:Q2L}

As in \S \ref{SS:FqL}, we calculate the kernels and cokernels of the reduced Adams operation. 

\begin{defin}\label{Def:KCQ2}
For all $k \geq 1$, let
\begin{equation*}
	\begin{aligned}
		K_2(k)\{x\} := &\bigoplus_{i \geq 0} \ker(\cdot 2^{\nu(k)+3} : \z/2^{s_3(i)}\{x v_1^{2k} \tau^{i}\} \to \z/2^{s_3(i)}\{xv_1^{2k} \tau^{i}\})
	\end{aligned}
\end{equation*}
\begin{equation*}
	\begin{aligned}
		C_2(k)\{x\} := &\bigoplus_{i \geq 0} \coker(\cdot 2^{\nu(k)+3} : \z/2^{s_3(i)}\{x v_1^{2k} \tau^i\} \to \z/2^{s_3(i)}\{x v_1^{2k}\tau^i\})\\
	\end{aligned}
\end{equation*}
be the sums of kernels and cokernels, respectively, for $x$ being either $z_i$ or $\rho^2$. 

Define 
$$K_2(k)=\begin{cases}
	\pi_{-2,*}^{\q_2}(H\z), & k=0\\
	K_2(k)\{\rho^2\}, & k>0
\end{cases}, \text{~~}
\underline K_2(k)=\begin{cases}
	\pi_{-1,*}^{\q_2}(H\z), & k=0\\
	K_2(k)\{z_i\}, & k>0
\end{cases}, \text{~and~}
C_2(k)=\begin{cases}
	\pi_{-2,*}^{\q_2}(H\z), & k=0\\
	C_2(k)\{\rho^2\}, & k>0
\end{cases}.
$$

We write $C_2(0)\{z_i\}=\pi_{-2,*}^{\q_2}(H\z)$ and define $\underline C_2(k), k\geq 0$ to be the sum
$$\underline C_2(k) := C_2(k)\{z_i\}\oplus \bigoplus_{i\geq 0} \z/2^{\nu(k)+3}\{y_iv_1^{2k}\tau^i\}\oplus \z/2^{\nu(k)+3}\{v_1^{2k}u\}.$$

Let $\tilde{K}_2(k)$ denote the groups obtained from $K_2(k)$ by replacing the $(4i-2)$-th summand with two times the summand, for all positive integer $i$; let $\tilde{\underline K}_2(k)$ denote the groups obtained from $K_2(k)$ by replacing each summand with two times the summand, except the $(4i-1)$-th summand for all positive integer $i$.

Let $\tilde{C}_2(k)$ denote the groups obtained from $C_2(k)$ by replacing the $(4i-2)$-th and $(4i-1)$-th summands with two times the summands, for all positive integer $i$; let $\tilde{\underline C}_2(k)$ denote the groups obtained from $\underline C_2(k)$ by replacing each summand with two times the summand.
\end{defin}

In the notation introduced in Definition \ref{Def:KCQ2}, we use $K$ for things from the kernel of the reduced Adams operation, and $C$ for things from the cokernel.
We use the underlined symbols to denote the groups coming from $\pi_{-1,*}^{\q_2}H\z$ and the no-underlined symbols to denote the groups coming from $\pi_{-2,*}^{\q_2}H\z$. The tilde is for the resulting groups after computing the differentials.
We use \Cref{Def:KCQ2} to describe the $E_1$-term and $E_\infty$-term of the ESSS for $L$. The $E_1$-term is depicted in \Cref{Fig:Q2E1L}.

In \Cref{Fig:Q2E1L}, the red differentials are given by taking $Sq^2$, the brown differentials are given by 
multiplying by $\tau$, and the blue differentials are $Sq^2+\rho Sq^1$. The calculation is similar to the case for $kq$ in \S \ref{SS:Q2kq}. The $E_2$-page is depicted in \Cref{Fig:Q2E2L}.

\begin{exm}\label{exm:Q2L}
We explain the $d_1$ calculation using sample computations. We suggest readers compare the following examples with Figure \ref{Fig:Q2E1L} and Figure \ref{Fig:Q2E2L}.
\begin{enumerate}
\item The differentials between the kernels (black symbols) from $(2,2)$ to $(1,5)$ is similar to Example \ref{exm:Q2kq}(2). The source contains $$K_2(1) = \bigoplus_{i \geq 0} \ker(\cdot 2^3 : \z/2^{s_3(i)}\{\rho^2 v_1^{2} \tau^{i}\} \to \z/2^{s_3(i)}\{\rho^2v_1^{2} \tau^{i}\}).$$
When $i=4k+3$, $s_3(i)>3$, and thus the generator of the $(4k+3)$-th summand is of the form $2^t\rho^2 v_1^{2} \tau^{4k+3}$ for some positive integer $t$. Therefore, the differential from this generator is zero, and $\tau^{4k}\rho^2 h_1^3$ is not hit.

After computing this differential, what remains in the target is $z/2[\tau^4]\{\rho^2 h_1^3\}, $ and what remains in the source after a base change is $\z/2\{1,\tau\}[\tau^4]\oplus \tilde{K}_2(1)$.

\item We explain the differential between the cokernels (darkgreen symbols) from $(2,2)$ to $(1,5)$. The source is 
$$\underline{C_2}(1)= \bigoplus_{i \geq 0} \coker(\z/2^{s_3(i)} \xrightarrow{\cdot 2^{3} } \z/2^{s_3(i)})\{z_i v_1^{2}\tau^i\}\oplus \bigoplus_{i\geq 0} \z/2^{3}\{y_iv_1^{2}\tau^i\}\oplus \z/2^{3}\{uv_1^{2}\}$$
 and the target is $\z/2[\tau]\{u,\pi,\rho\}$. View $\underline{C_2}(1)$ as a $\z$-module, the generators are of the form 
 $\rho v_1^2 \tau^{2n} $, $\pi v_1^2 \tau^{2n}  $ , $\pi v_1^2 \tau^{2n+1} $, $u v_1^2\tau^{2n+1} $  and $u v_1^2$ for $n\geq 0, $ so what remains in the target are $\z/2\{\rho v_1^2,\pi v_1^2,u v_1^2\}\oplus \z/2\{\rho v_1^2\tau^{2n}, n\geq 0\}\oplus \z/2\{u v_1^2\tau^{2n+1},  n\geq 0\}$.
\end{enumerate}
	
\end{exm}

Every potential higher differential hits an $\eta$-periodic class, or is from the unit $1$. 
By comparing with the $\eta$-inverted result  
in \cite[Cor. 11]{Wil18}, there is no possibility for nontrivial higher differentials. Therefore, over $\q_2$, the ESSS for $L$ collapses at $E_2$.

\subsection{$\mathbb{R}$}\label{SS:RL}

The ESSS for $L$ over $\r$ was computed in \cite[Sec. 5]{BIK22}. The multiplicative generators of the $E_1$-term appear in \cite[Table 8]{BIK22} and the $E_1$-term is depicted in \cite[Fig. 10]{BIK22}. The values of the $d_1$-differentials on the multiplicative generators appear in \cite[Table 10]{BIK22}, and these values determine all of the $d_1$-differentials using the Leibniz rule in conjunction with the relations recorded in \cite[Table 9]{BIK22}. The higher differentials are described in \cite[Prop. 5.11]{BIK22}. The resulting $E_\infty$-page is depicted in \cite[Figs. 13-19]{BIK22}.

\subsection{$\mathbb{Q}$}\label{SS:QL}

To compute $E_1^{\q}(L)$, we break it into four manageable pieces. First, there is a decomposition
$$E_1^{\q}(L) \cong K^{\q}(L) \oplus \Sigma^{-1} C^{\q}(L),$$
where $K^{\q}(L)$ is the kernel of $\psi^3-1$ and $C^{\q}(L)$ is the cokernel of $\psi^3-1$. There are further decompositions
$$K^{\q}(L) \cong K^+(L) \oplus K^-(L),$$
$$C^{\q}(L) \cong C^+(L) \oplus C^-(L),$$
where 
$$K^+(L) := \ker(\psi^3-1 : E_1^+(kq) \to E_1^+(kq)),$$
$$K^-(L) := \ker(\psi^3-1 : E_1^-(kq) \to E_1^-(kq)),$$
$$C^+(L) := \coker(\psi^3-1:E_1^+(kq) \to E_1^+(kq)),$$
$$C^-(L) := \coker(\psi^3-1:E_1^-(kq) \to E_1^-(kq)).$$
Here, $E_1^+(kq)$ and $E_1^-(kq)$ were defined in \Cref{SS:Qkq}. We then have
$$E_1^{\q}(L) \cong K^+(L) \oplus K^-(L) \oplus \Sigma^{-1} C^+(L) \oplus \Sigma^{-1} C^-(L).$$

Since $E_1^+(kq)$ can be expressed entirely in terms of $E_1^{\f_q}(kq)$, $q$ odd, the groups $K^+(L)$ and $C^+(L)$ can be described using the computations from \Cref{SS:FqL}. More precisely, $K^+(L) \oplus \Sigma^{-1} C^+(L)$ is obtained graphically by taking the sum over all odd primes of \Cref{Fig:FqE1L}, then shifting the entire picture by $(-1,1)$. 

The computation of $K^-(L)$ and $C^-(L)$ also largely follows from previous computations. Since $E_1^{\r}(kq)$ is a summand of $E_1^-(kq)$, the kernel and cokernel of $\psi^3-1$ over $\r$ are summands in $K^-(L) \oplus \Sigma^{-1} C^-(L)$; the relevant groups are depicted graphically in \cite[Fig. 10]{BIK22}. The complementary summand of $E_1^-(kq)$ is obtained by replacing $\pi_{**}(H\z)$ by $C'''(0)$ and $\pi_{**}(H\z/2)$ by $\z/2[\tau]\{[2]\}$. Since $\psi^3-1$ is necessarily trivial on $\z/2[\tau]\{[2]\}$, we only need to analyze its effect on each copy of $C'''(0)$, but this was already done in \Cref{SS:Q2L} since $C'''(0)$ may be identified with the summand in $\pi_{**}^{\q_2}(H\z)$ generated by classes of the form $\pi \tau^m$, $m \geq 0$, or equivalently, by the classes of the form $y_m \tau^m$ for $m$ even and $z_m \tau^m$ for $m$ odd.  Note that any nontrivial $d_1$-differential over $\q_2$ involving $\pi \tau^m$ must include it in both the source and target, so all such $d_1$-differentials lift to $\q$.

The preceding discussion implies that the map
$$E_1^{\q}(L) \to \prod_{\nu} E_1^{\q_\nu}(L)$$
is injective. The $d_1$-differentials on $E_1^+(L)$ are precisely the lifts of the $d_1$-differentials from \Cref{SS:FqL} (which occur on the $\pi$-divisible part of the ESSS over $\q_q$, $q$ odd). The $d_1$-differentials on $E_1^-(L)$ are the lifts of the $d_1$-differentials over $\r$ (\cite[Table 10]{BIK22}) and the lifts of the $d_1$-differentials over $\q_2$ for which the source (equivalently, the target) lift to $\q$. 

Since all of the $d_1$-differentials over $\q$ are lifted from $d_1$-differentials over the local places, the map
$$E_2^{\q}(L) \to \prod_{\nu} E_2^{\q_\nu}(L)$$
is also injective. There cannot be any higher differentials involving $E_2^+(L)$, since these would imply higher differentials in $E_2^{\q_q}(L)$ for $q$ odd. On $E_2^-(L)$, or more generally on $E_r^-(L)$ for $r \geq 2$, the nontrivial differentials are precisely the lifts of the nontrivial $d_r$-differentials in $E_r^{\r}(L)$ described in \cite[Prop. 5.11]{BIK22}. 

The $E_\infty$-term can be described as follows. We have 
$$E_\infty^{\q}(L) \cong E_\infty^+(L) \oplus E_\infty^-(L).$$
The piece $E_\infty^+(L)$ is the sum over all odd primes $q$ of $E_\infty^{\f_q}(L)$ (see \Cref{Fig:FqEooL}), shifted by $(-1,1)$. The piece $E_\infty^-(L)$ is the sum of $E_\infty^{\r}(L)$ (\cite[Figs. 13-19]{BIK22}) and the portion of $E_\infty^{\q_2}(L)$ (\Cref{Fig:Q2E2L}) generated by $\pi \tau^m$ and $\iota \pi \tau^m$, $m \geq 0$. All hidden extensions follow from comparison with the local places.  

\begin{rem2}
Since the Hasse map for $L$ is injective on the $E_1$-term of the ESSS and every differential over $\q$ is lifted from a differential over some $\q_\nu$, the Hasse map
$$\pi_{**}^{\q}(L) \to \prod_{\nu} \pi_{**}^{\q_\nu}(L)$$
is injective, i.e., $L$ satisfies the motivic Hasse principle of \cite[Sec. 4]{OO13}. 
\end{rem2}

\clearpage 

\appendix

\section{Figures}\label{Sec:Figures}

\begin{figure}[!htb]
\centering
\includegraphics[trim=0 30 0 20, clip, scale=0.7]{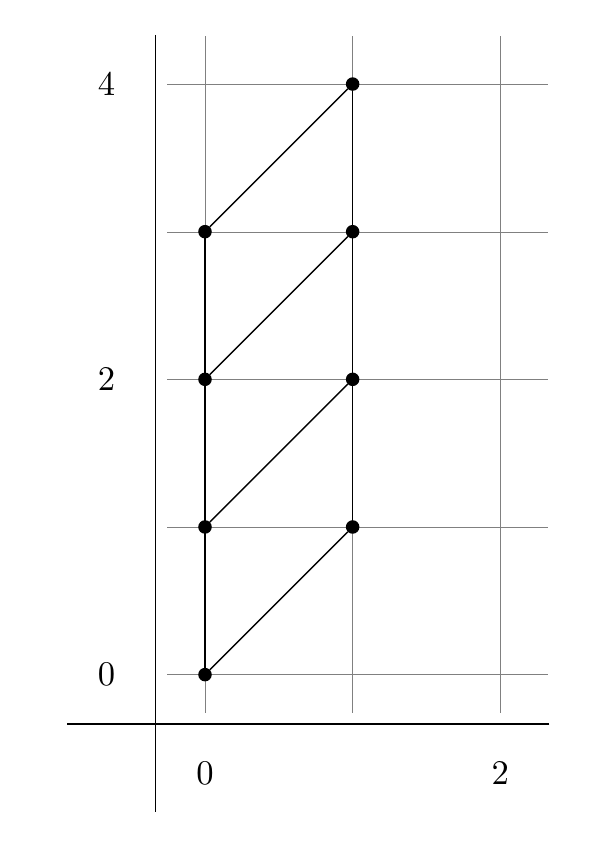}
\includegraphics[trim=0 30 0 20, clip, scale=0.7]{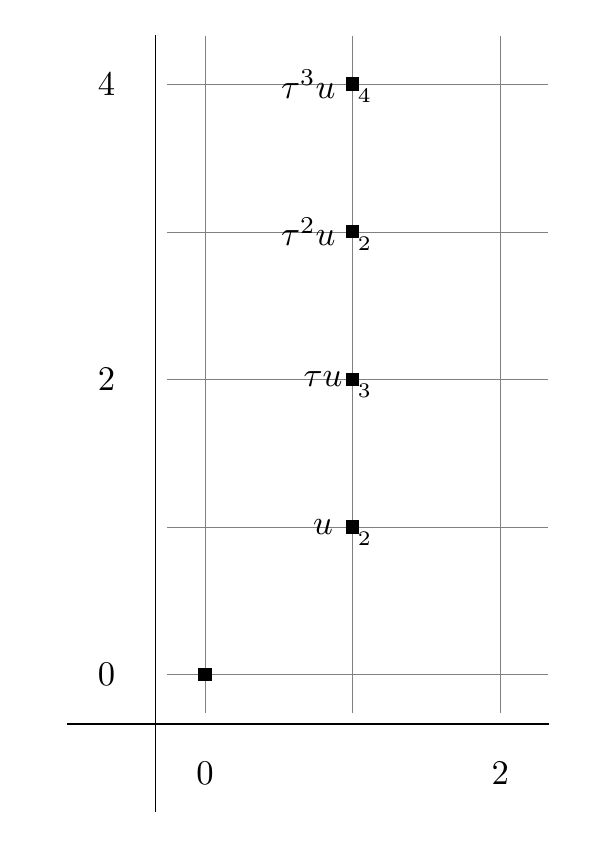}
\includegraphics[trim=0 30 0 10, clip, scale=0.7]{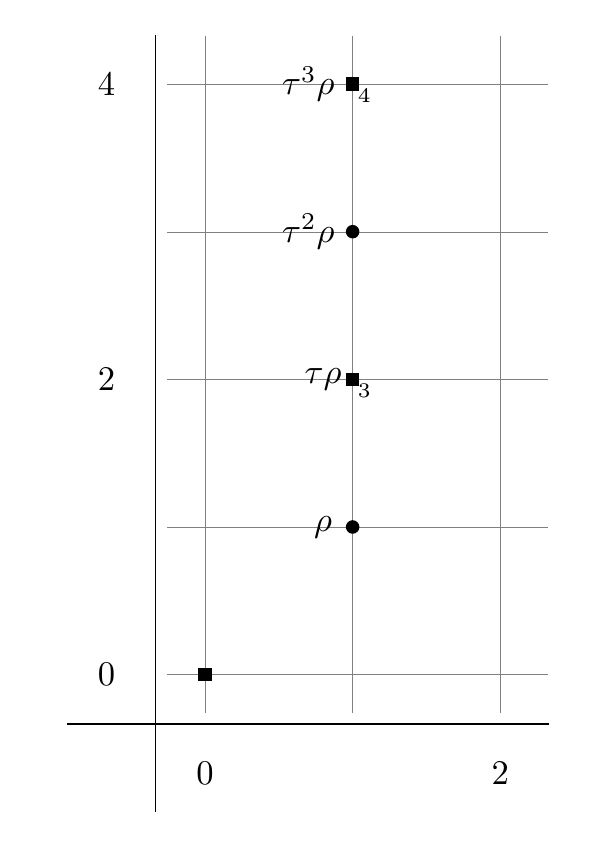}
\includegraphics[trim=0 30 0 10, clip, scale=0.7]{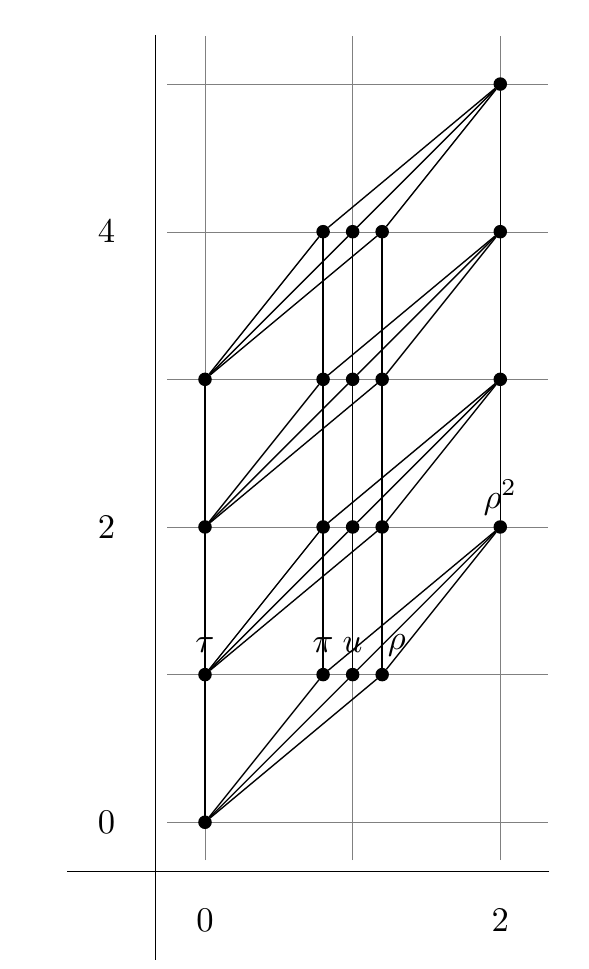}
\includegraphics[trim=0 30 0 10, clip, scale=0.7]{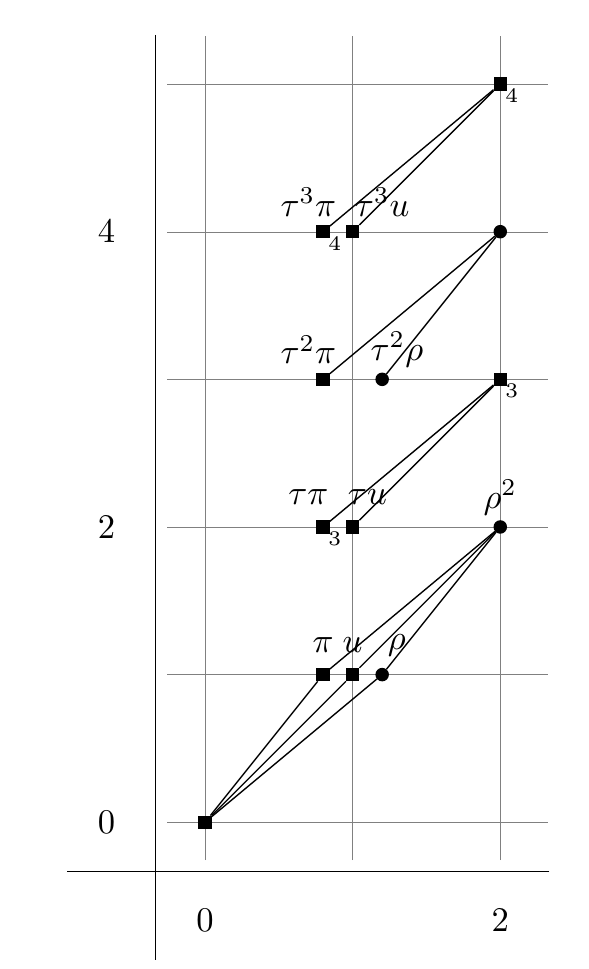}
\includegraphics[trim=0 30 0 10, clip, scale=0.7]{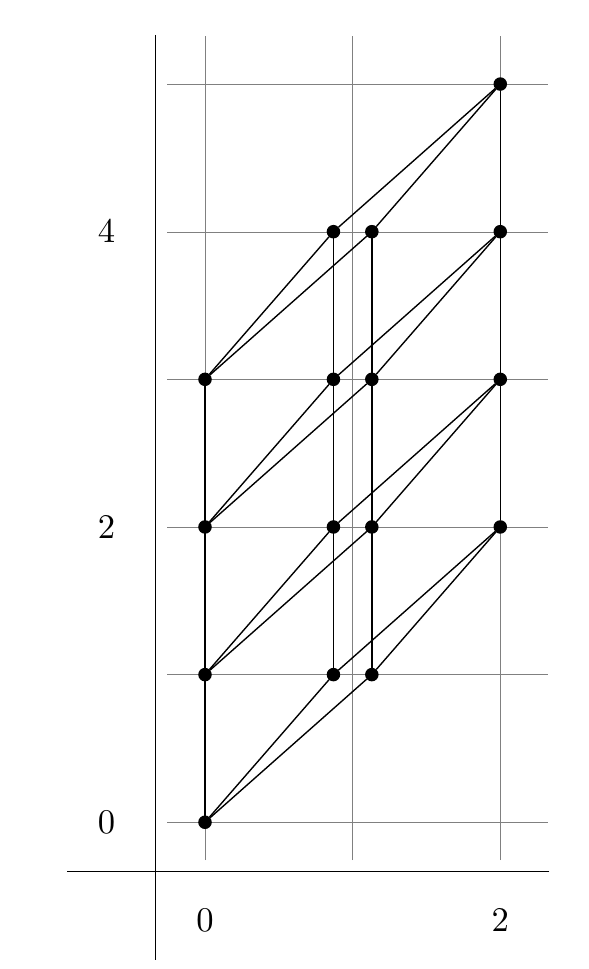}
\includegraphics[scale=0.7]{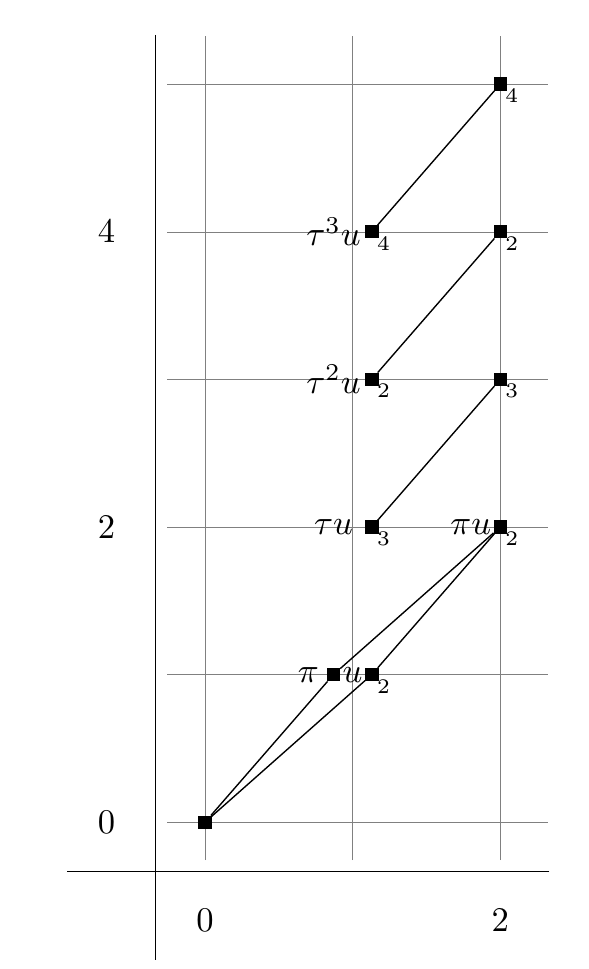}
\includegraphics[scale=0.7]{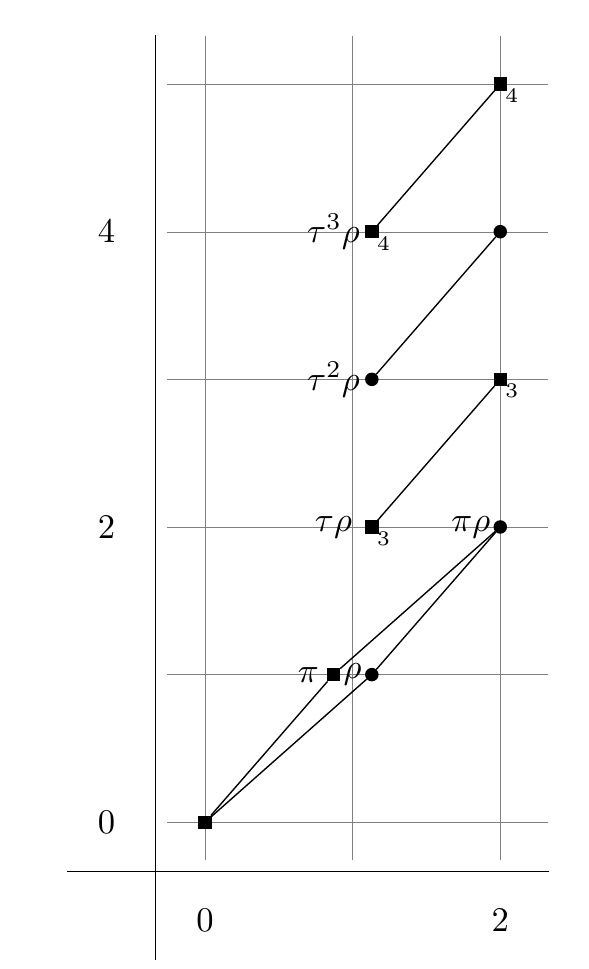}
\caption{A graphical depiction of $\pi_{**}^{\f_q}(H\z/2)$, $\pi_{**}^{\f_5}(H\z)$, $\pi_{**}^{\f_3}(H\z)$, $\pi_{**}^{\q_2}(H\z/2)$, $\pi_{**}^{\q_2}(H\z)$, $\pi_{**}^{\q_q}(H\z/2)$, $\pi_{**}^{\q_5}(H\z)$, and $\pi_{**}^{\q_3}(H\z)$. A bullet $\bullet$ represents $\z/2$, a black square $\blacksquare$ represents $\z$, a black square with a number {$\blacksquare_n$} represents $\z/{2^n}$. The $x$ axis is $-s$ and the $y$ axis is $-w.$ }\label{Fig:EM_Qq}
\end{figure}

\begin{figure}[!htb]
\centering
\includegraphics[scale=1.25]{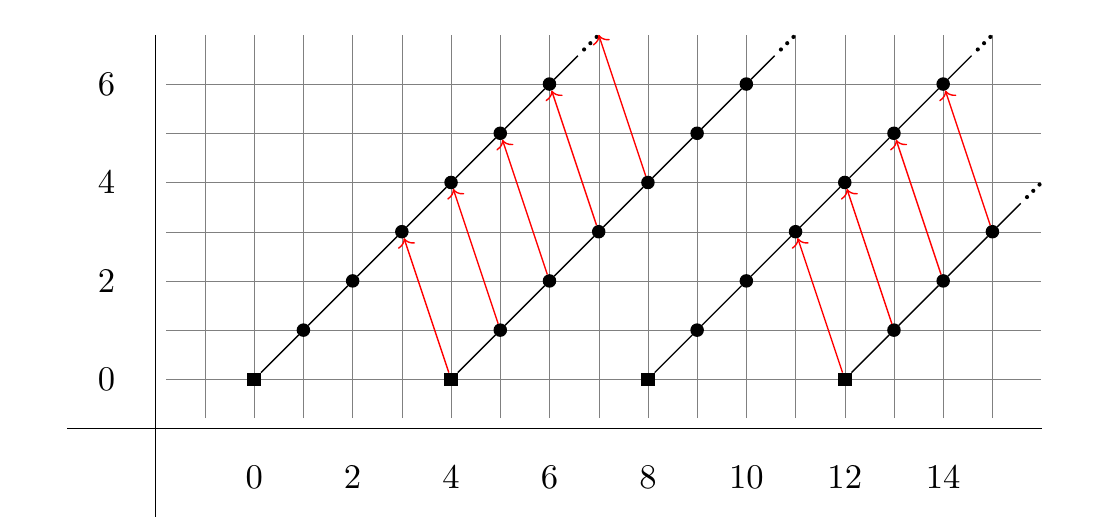}
\caption{The $E_1$-page of the ESSS for $kq$ over an algebraically closed field. A bullet $\bullet$ represents $\z/2[\tau]$ and a square $\blacksquare$ represents $\z[\tau]$. This figure appears with generators labeled as \cite[Fig. 1]{BIK22}.}\label{Fig:CE1kq}
\end{figure}

\begin{figure}[!htb]
\centering
\includegraphics[scale=1.25]{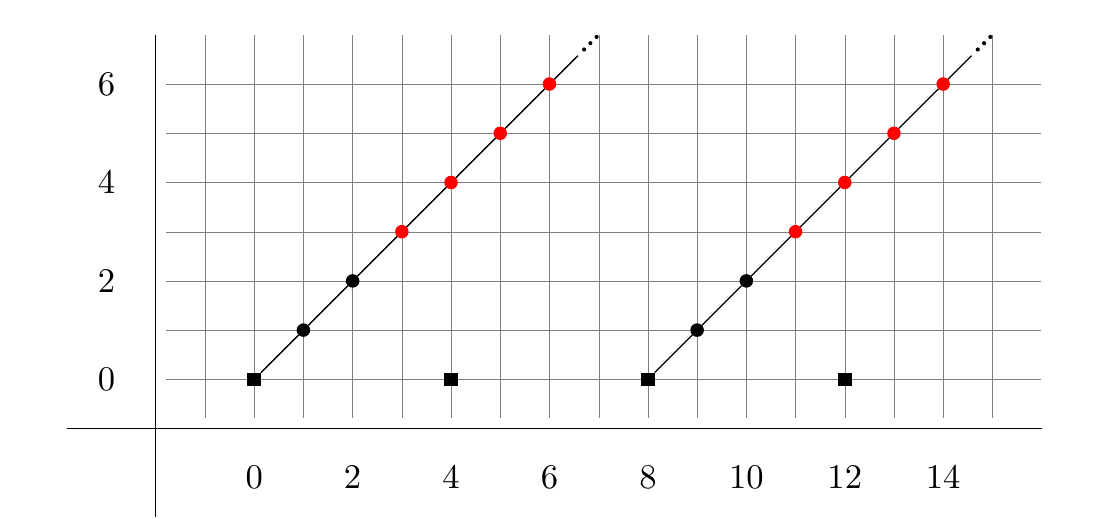}
\caption{The $E_\infty$-page of the ESSS for $kq$ over an algebraically closed field. A black bullet $\bullet$ represents $\z/2[\tau]$, a red bullet \textcolor{red}{$\bullet$} represents $\z/2$, and a square $\blacksquare$ represents $\z[\tau]$. This figure appears with generators labeled as \cite[Fig. 2]{BIK22}.}\label{Fig:CEookq}
\end{figure}

\begin{figure}[!htb]
\centering
\includegraphics[scale=1.25]{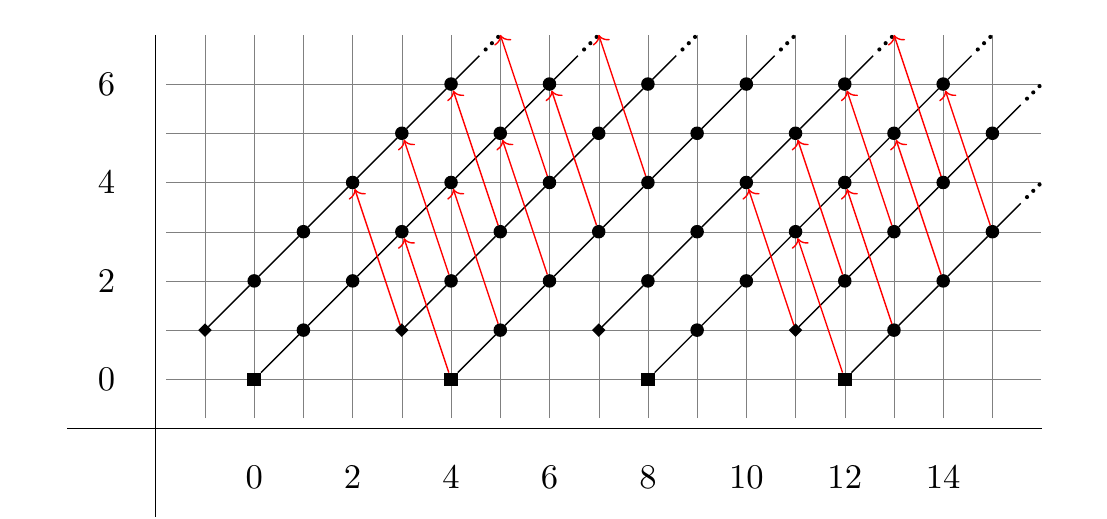}
\caption{The $E_1$-page of the ESSS for $kq$ over $\f_q$. A bullet $\bullet$ represents $\z/2[\tau]$, a square $\blacksquare$ represents $\z$, and a diamond $\blackdiamond$ represents the $u$-divisible (if $q \equiv 1 \mod 4$) or $\rho$-divisible (if $q \equiv 3 \mod 4$) part of $\pi_{**}^{\f_q}(H\z)$.}\label{Fig:FqE1kq}
\end{figure}

\begin{figure}[!htb]
\centering
\includegraphics[scale=1.25]{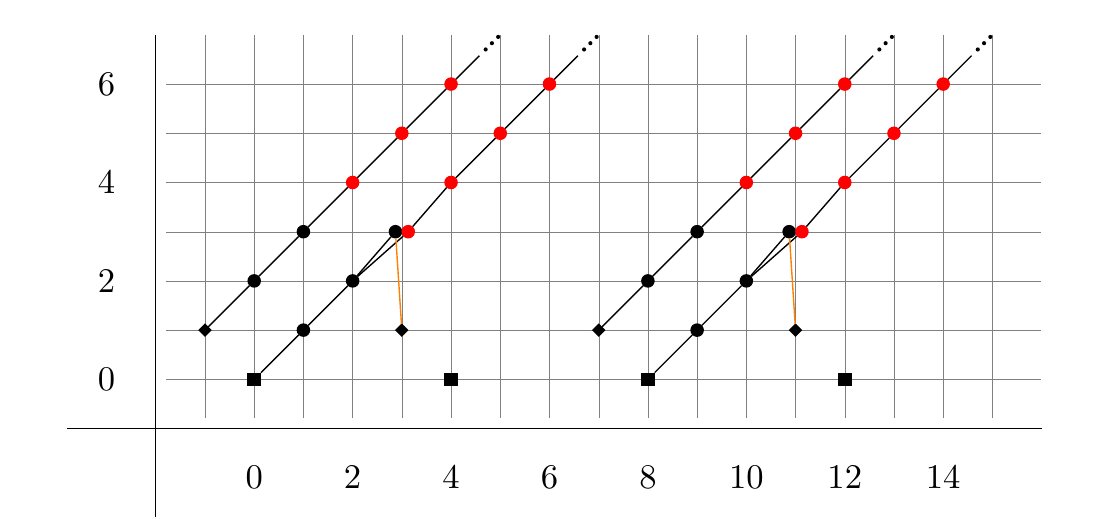}
\caption{The $E_\infty$-page of the ESSS for $kq$ over $\f_q$. A black bullet $\bullet$ represents $\z/2[\tau]$, a red bullet \textcolor{red}{$\bullet$} represents $\z/2$, a square $\blacksquare$ represents $\z$, and a diamond $\blackdiamond$ represents $\bigoplus_{i \geq 0} \z/2^{\nu(q-1)+\nu(i+1)-1}\{2 \tau^i\}$. The orange lines indicate hidden $\mathsf h$-extensions.}\label{Fig:FqEookq}
\end{figure}

\begin{figure}[!htb]
\centering
\includegraphics[scale=.75]{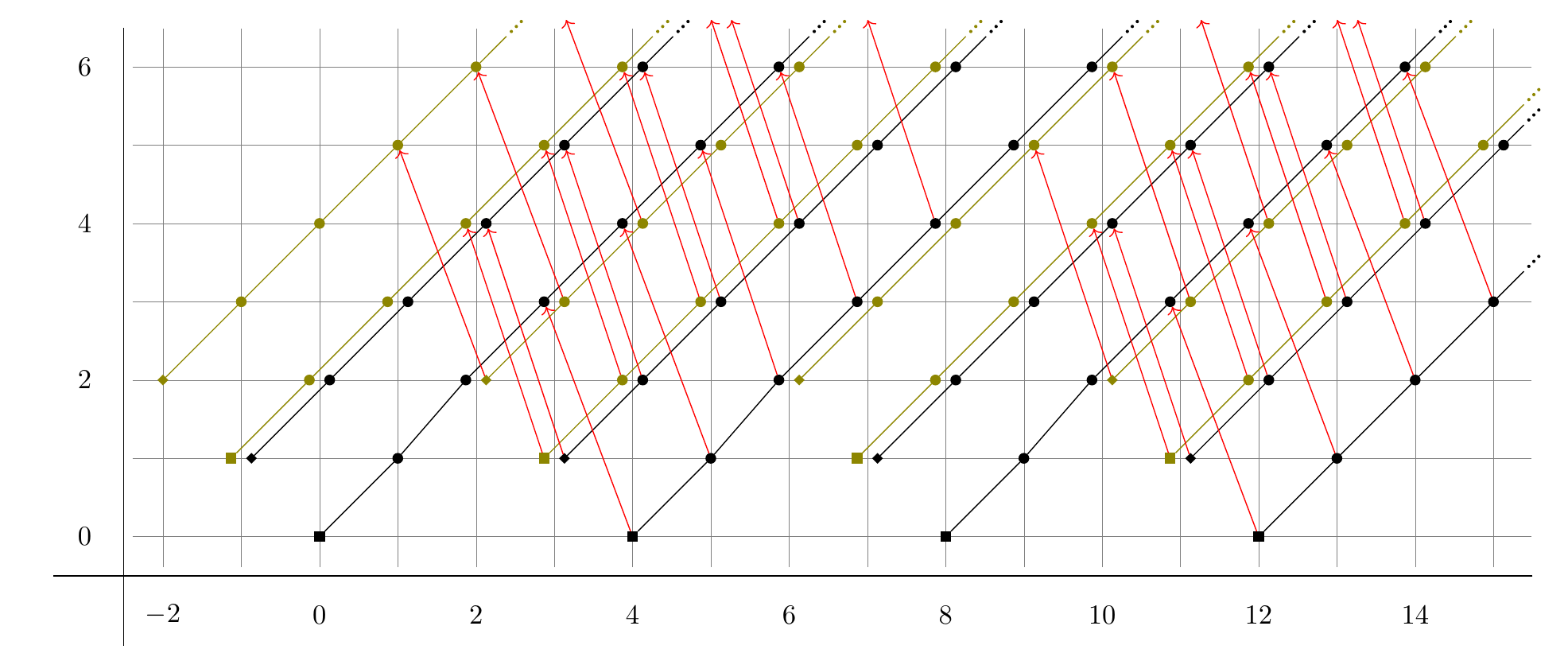}
\caption{The $E_1$-page of the ESSS for $kq$ over $\q_q$. A bullet $\bullet$ represents $\z/2[\tau]$, a square $\blacksquare$ represents $\z$, and a diamond $\blackdiamond$ represents the $u$-divisible (if $q \equiv 1 \mod 4$) or $\rho$-divisible (if $q \equiv 3 \mod 4$) part of $\pi_{**}^{\q_q}(H\z)$. Classes in olive are $\pi$-divisible.}\label{Fig:QqE1kq}
\end{figure}

\begin{figure}[!htb]
\centering
\includegraphics[scale=.75]{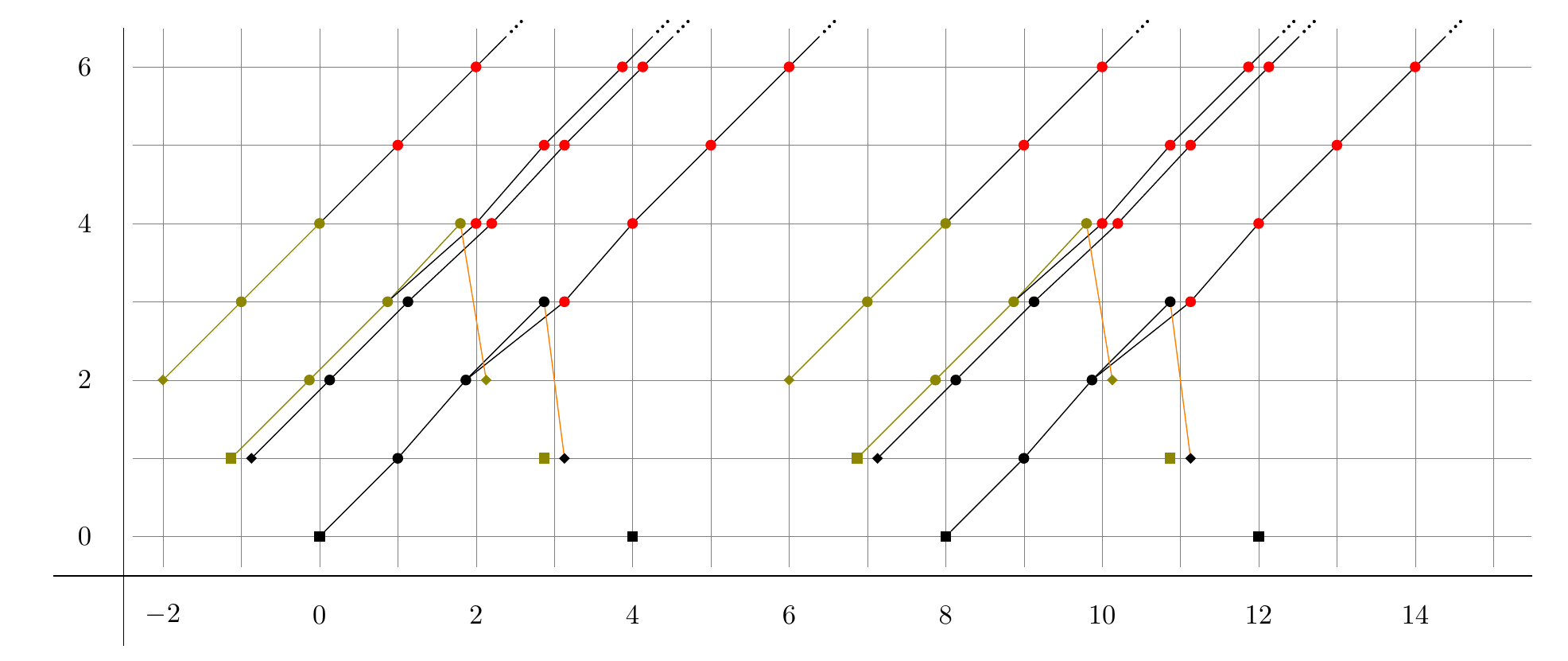}
\caption{The $E_\infty$-page of the ESSS for $kq$ over $\q_q$. A black bullet $\bullet$ represents $\z/2[\tau]$, a red bullet \textcolor{red}{$\bullet$} represents $\z/2$, a square $\blacksquare$ represents $\z$, and a diamond $\blackdiamond$ represents the group obtained from the $u$-divisible (if $q \equiv 1 \mod 4$) or $\rho$-divisible (if $q \equiv 3 \mod 4$) part of $\pi_{**}^{\f_q}(H\z)$ by dividing the order by two. Classes in olive are $\pi$-divisible, and orange lines indicate hidden $\mathsf h$-extensions.}\label{Fig:QqEookq}
\end{figure}

\begin{figure}[!htb]
\centering
\includegraphics[scale=.75]{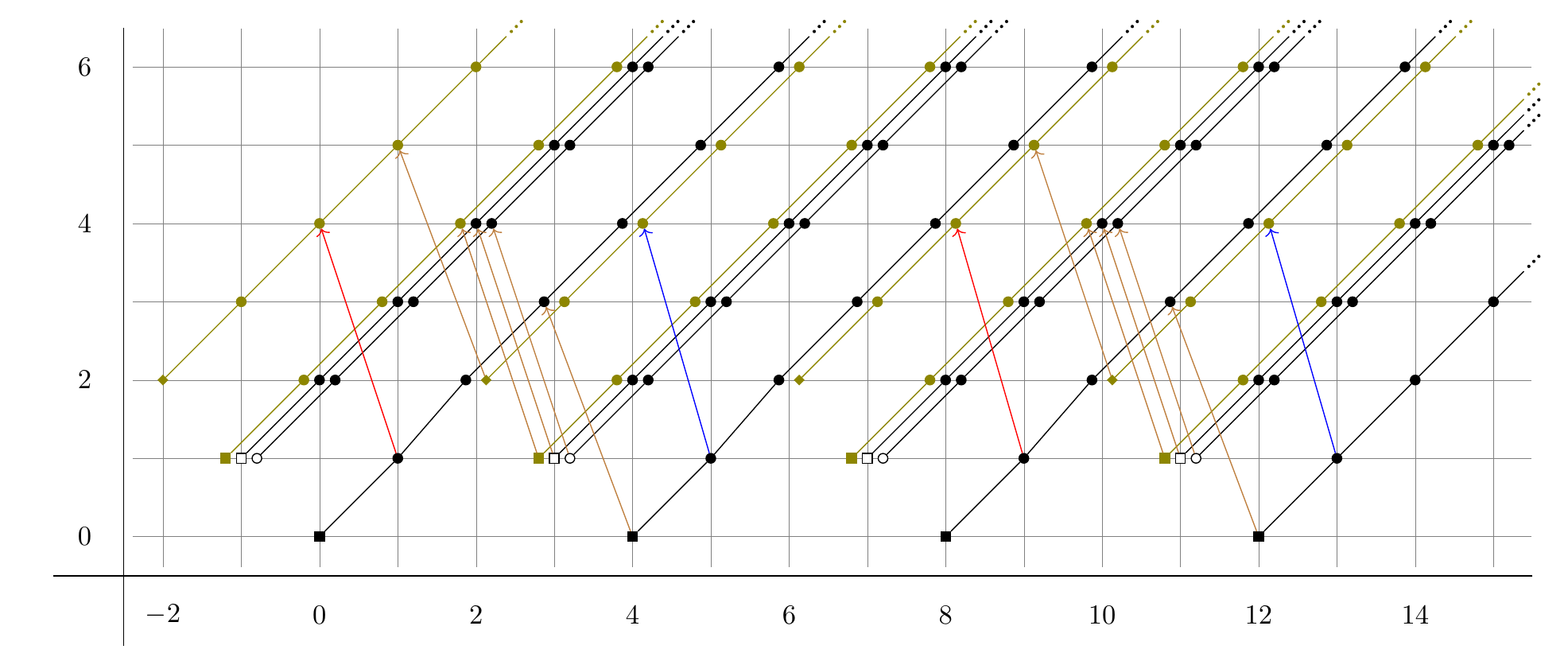}
\caption{The $E_1$-page of the ESSS for $kq$ over $\q_2$. A bullet $\bullet$ represents $\z/2[\tau]$, a circle $\circ$ represents $\z/2[\tau^2]\{\rho\}$, a black square $\blacksquare$ represents $\z$, a rectangle $\square$ represents $\z[\tau^2]\{\tau u\}\oplus \z\{u\}$, a diamond ${\color{olive} \blackdiamond}$ represents $\pi_{-2,*}^{\q_2}(H\z)$, and an olive square ${\color{olive} \blacksquare}$ represents the $\pi$-divisible part of $\pi_{-1,*}^{\q_2}(H\z)$. Classes in olive are $\pi$-divisible. The differentials are $h_1$-periodic; we only draw the first occurrences.}\label{Fig:Q2E1kq}
\end{figure}

\begin{figure}[!htb]
\centering
\includegraphics[scale=0.75]{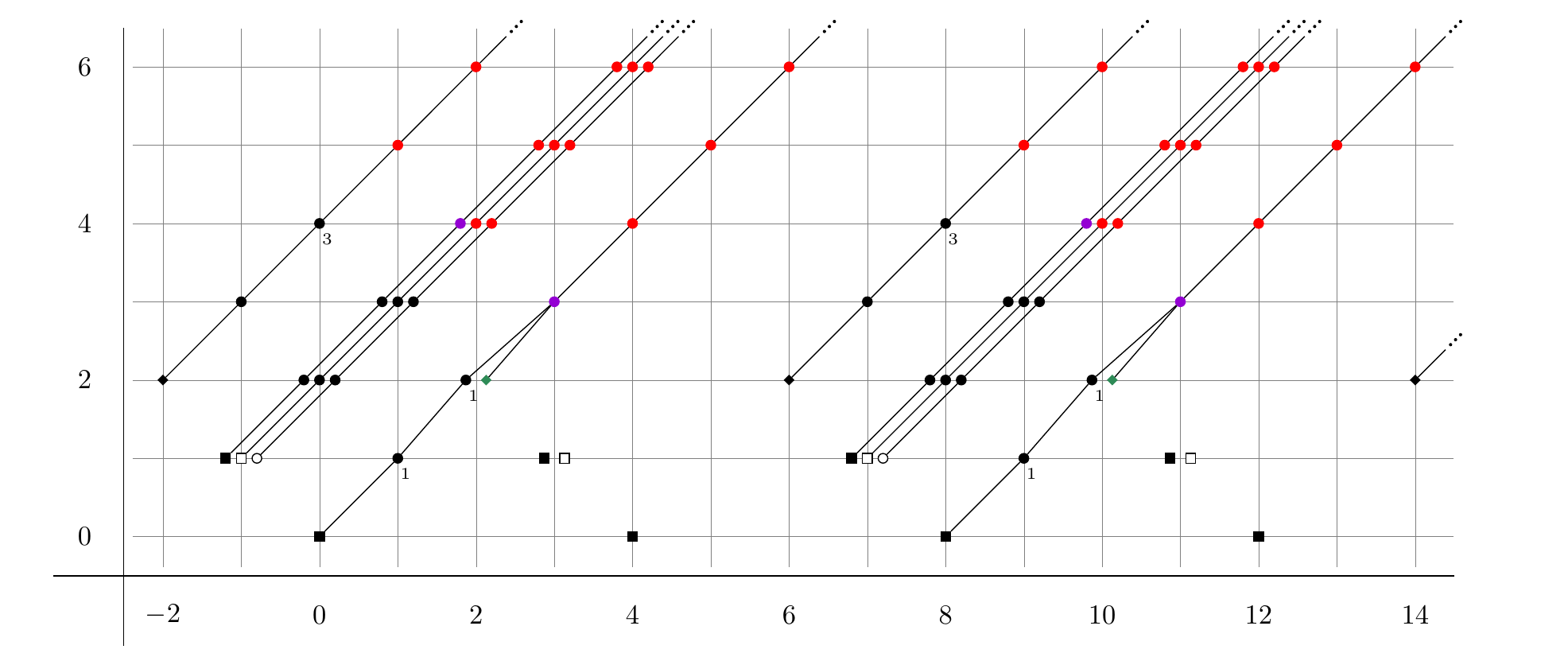}
\caption{The $E_\infty$-page of the ESSS for $kq$ over $\q_2$. 
A bullet $\bullet$ represents $\z/2[\tau]$, a red bullet \textcolor{red}{$\bullet$} represents $\z/2$, a bullet with subscript $\bullet_n$ represents $\z/2[\tau^4]\{1,\tau^n\}$, 
a lavender bullet \hana{$\bullet$} represents $\z/2[\tau]\{\tau^2\}\oplus \z/2$,
a circle $\circ$ represents $\z/2[\tau^2]\{\rho\}$,
a square $\blacksquare$ represents $\z$, a rectangle $\square$ represents $\z[\tau^2]\{\tau u\}\oplus \z\{u\}$, 
a diamond ${\blackdiamond}$ represents $\pi_{-2,*}^{\q_2}(H\z)$,
and a dark green diamond ${\color{seaolive} \blackdiamond}$ represents replacing each group of $\pi_{-2,*}^{\q_2}(H\z)$ in degree $(-2,-4k)$ and $(-2,-1-4k)$ with twice of it for all positive integers $k$.
}\label{Fig:Q2E2kq}  
\end{figure}

\begin{figure}[!htb ]
\centering
\includegraphics{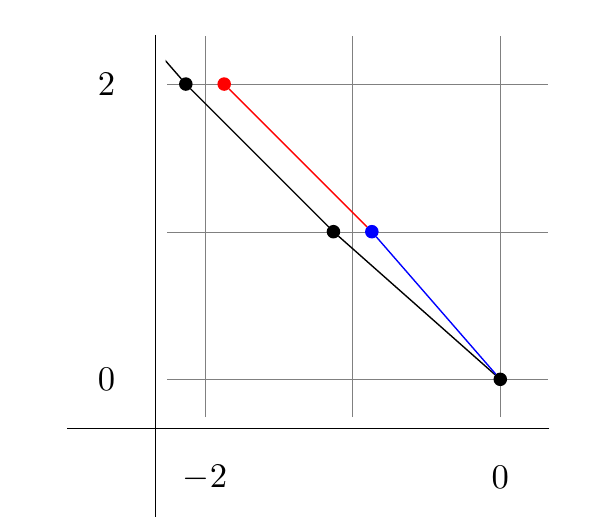}
\includegraphics{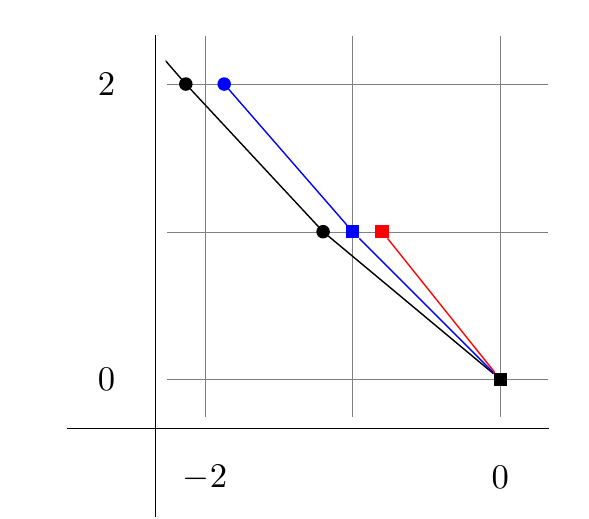}
\caption{A graphical depiction of $\pi_{**}^{\q}(H\z/2)$ (left) and $\pi_{**}^{\q}(H\z)$ (right). In the left-hand picture, a black bullet $\bullet$ represents $\z/2[\tau]$, a blue bullet \textcolor{blue}{$\bullet$} represents $\bigoplus_{q \text{ prime}} \z/2[\tau]\{[q]\}$, and a red bullet \textcolor{red}{$\bullet$} represents $\bigoplus_{q \text{ odd}} \z/2[\tau]\{a_q\}$. In the right-hand picture, a black square $\blacksquare$ represents $\z$, a black bullet $\bullet$ represents $\z/2[\tau^2]$, a blue square \textcolor{blue}{$\blacksquare$} represents $\bigoplus_{q \text{ odd}} \z\{[q]\}$, a red square \textcolor{red}{$\blacksquare$} represents $C'''(0)$, and a blue bullet \textcolor{blue}{$\bullet$} represents $\bigoplus_{q \text{ odd}} \bigoplus_{i \geq 0} \z/2^{s_q(i)}\{[q]x_q \tau^i\}$.}\label{Fig:QF2Z}
\end{figure}

\begin{figure}[!htb]
\centering
\includegraphics[scale=1.25]{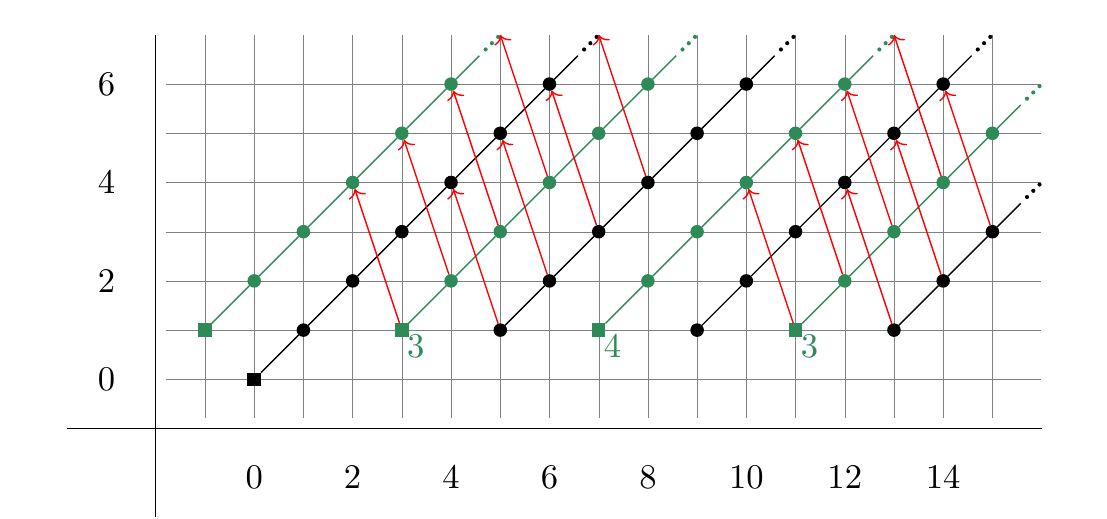}
\caption{The $E_1$-page of the ESSS for $L$ over algebraically closed fields. A bullet $\bullet$ represents $\z/2[\tau]$, a square $\blacksquare$ represents $\z[\tau]$, and a square with the positive integer $n$ as its right subscript $\blacksquare_n$ represents $\z/2^n[\tau]$. Elements of the form $\iota x$ appear in dark green. This figure appears with generators labeled as \cite[Fig. 3]{BIK22}.}\label{Fig:CE1L}
\end{figure}

\begin{figure}[!htb]
\centering
\includegraphics[scale=1.25]{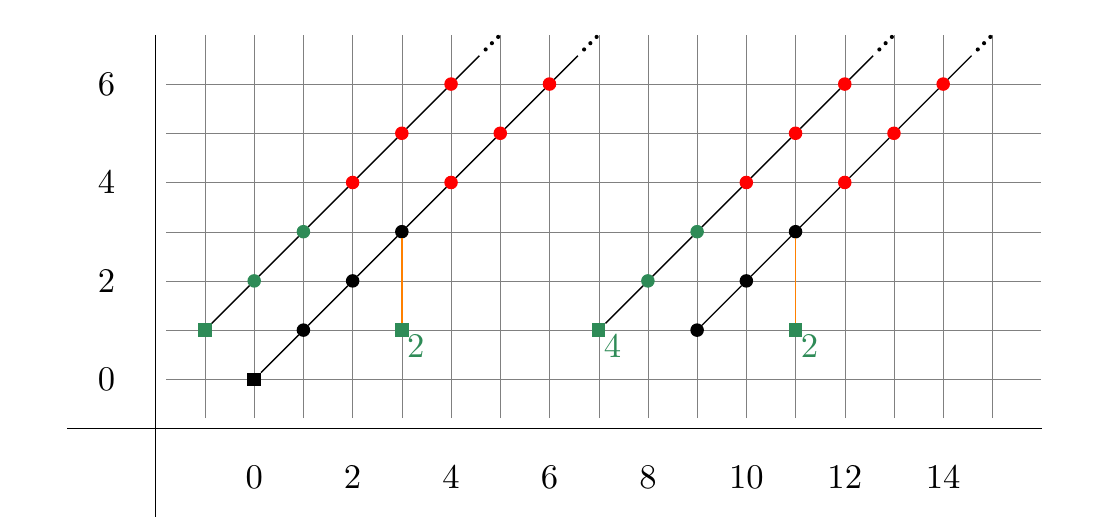}
\caption{The $E_\infty$-page of the ESSS for $L$ over algebraically closed fields. A bullet $\bullet$ represents $\z/2[\tau]$,  a red bullet \textcolor{red}{$\bullet$} represents $\z/2$, a square $\blacksquare$ represents $\z[\tau]$, and a square with the positive integer $n$ as its right subscript $\blacksquare_n$ represents $\z/2^n[\tau]$. Hidden extensions are depicted with orange lines. This figure appears with generators labeled as \cite[Fig. 3]{BIK22}.}\label{Fig:CEooL}
\end{figure}

\begin{figure}[!htb]
\centering
\includegraphics[scale=.75]{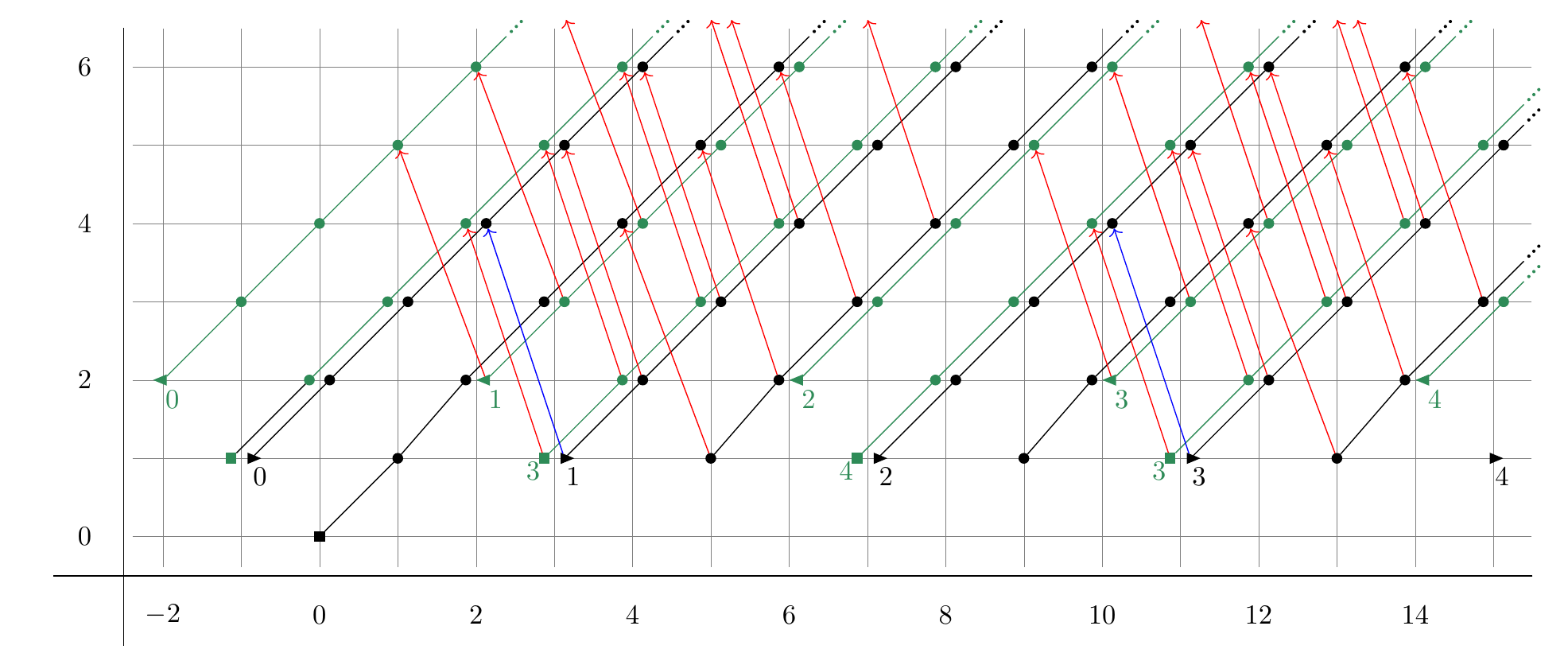}
\caption{The $E_1$-page of the ESSS for $L$ over $\f_q$. A bullet $\bullet$ represents $\z/2[\tau]$, a square $\blacksquare$ represents $\z$, a square with the positive integer $n$ as its left subscript ${}_n\blacksquare$ represents $\z/2^n$, an isosceles triangle with apex pointing right and nonnegative integer $n$ as its right subscript $\blacktriangleright_n$ represents $K_q(n)$, and an isosceles triangle with apex pointing left and nonnegative integer $n$ as its right subscript $\blacktriangleleft_n$ represents $C_q(n)$. Elements of the form $\iota x$ appear in dark green.}\label{Fig:FqE1L}
\end{figure}

\begin{figure}[!htb]
\centering
\includegraphics[scale=.75]{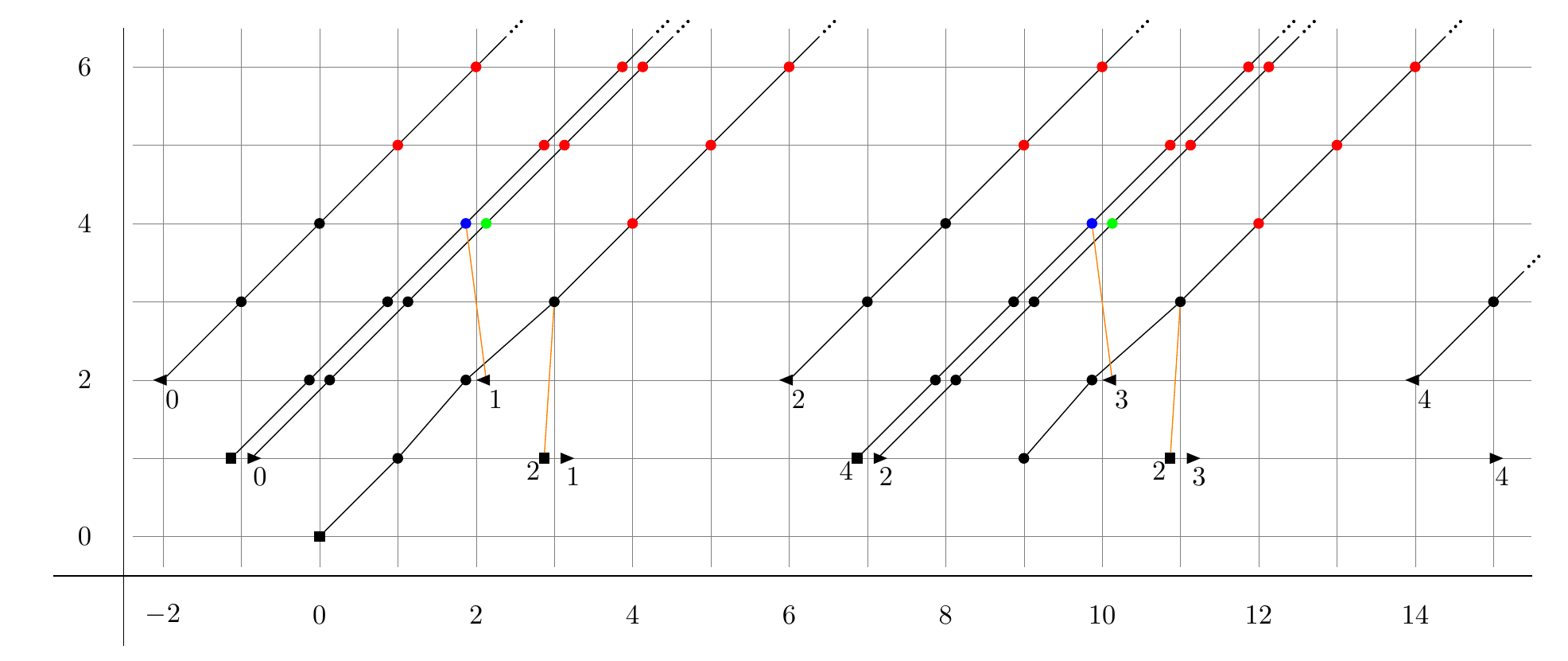}
\caption{The $E_\infty$-page of the ESSS for $L$ over $\f_q$. A bullet $\bullet$ represents $\z/2[\tau]$, a red bullet \textcolor{red}{$\bullet$} represents $\z/2$, a blue bullet \textcolor{blue}{$\bullet$} represents $\z/2\{1\} \oplus \z/2[\tau]\{\tau^2\}$, a green bullet \textcolor{green}{$\bullet$} represents the cokernel of the corresponding green $d_1$-differential, a square $\square$ represents $\z$, a square with the positive integer $n$ as its left subscript represents $\z/2^n$, an isosceles triangle with apex pointing right and nonnegative integer $n$ as its right subscript $\blacktriangleright_n$ represents $\tilde{K}_q(n)$, and an isosceles triangle with apex pointing left and nonnegative integer $n$ as its right subscript $\blacktriangleleft_n$ represents $\tilde{C}_q(n)$. Hidden extensions are depicted with orange lines.}\label{Fig:FqEooL}
\end{figure}

\begin{figure}[!htb]
\centering
\includegraphics[scale=.75]{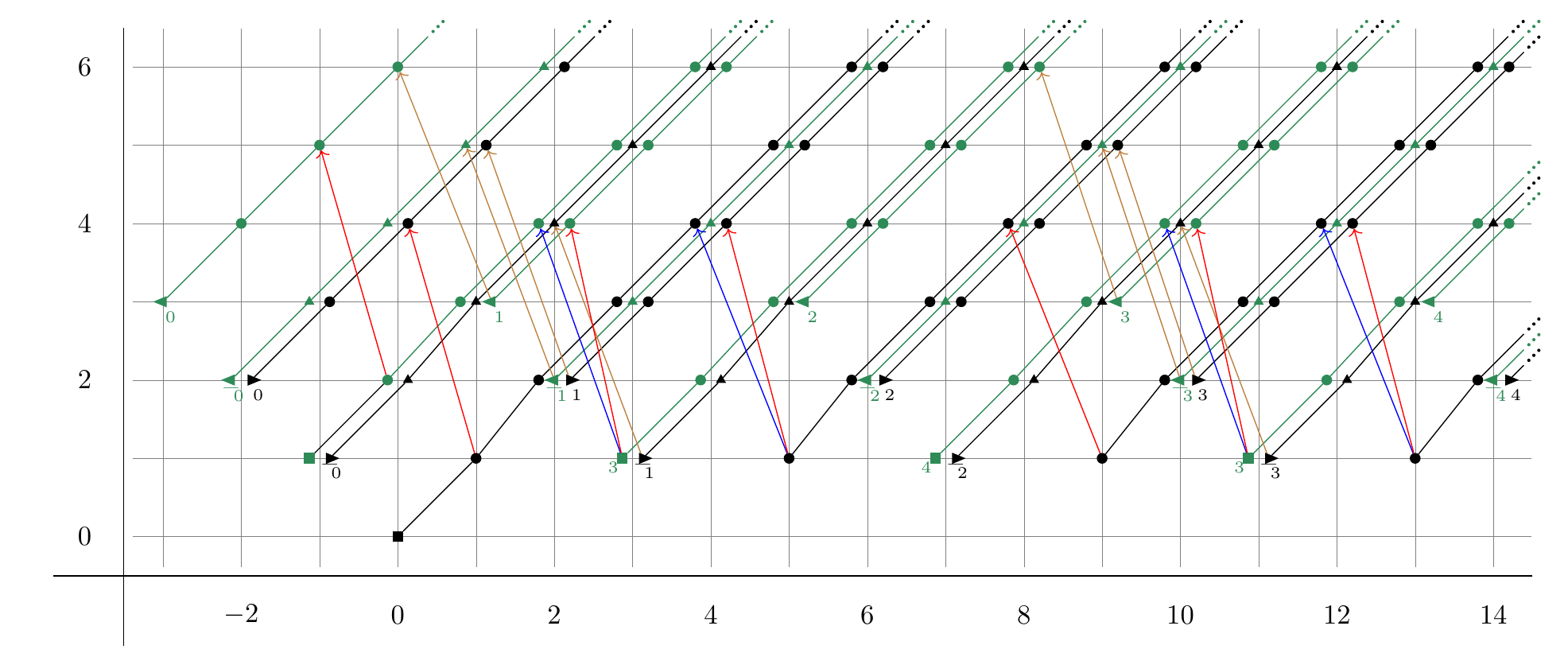}
\caption{ 
The $E_1$-page of the ESSS for $L$ over $\q_2$. A bullet $\bullet$ represents $\z/2[\tau]$, 
a triangle $\blacktriangle$ represents $(\z/2[\tau])^{3}$,
a square $\blacksquare$ represents $\z$, 
a square with the positive integer $n$ as its left subscript ${}_n\blacksquare$ represents $\z/2^n$, 
an isosceles triangle with apex pointing left and nonnegative integer $n$ as its right subscript $\blacktriangleleft_n$ ($\underline \blacktriangleleft_n$) represents $C_2(n)$ ($\underline C_2(n)$), 
and an isosceles triangle with apex pointing right and nonnegative integer $n$ as its right subscript $\blacktriangleright_n$ ($\underline \blacktriangleright_n$) represents ${K}_2(n)$ ($\underline{K}_2(n)$).
Elements of the form $\iota x$ appear in dark green.
The differentials are $h_1$-periodic; we only draw the first occurrences.
}\label{Fig:Q2E1L}
\end{figure}

\begin{figure}[!htb]
\centering
\includegraphics[scale=.75]{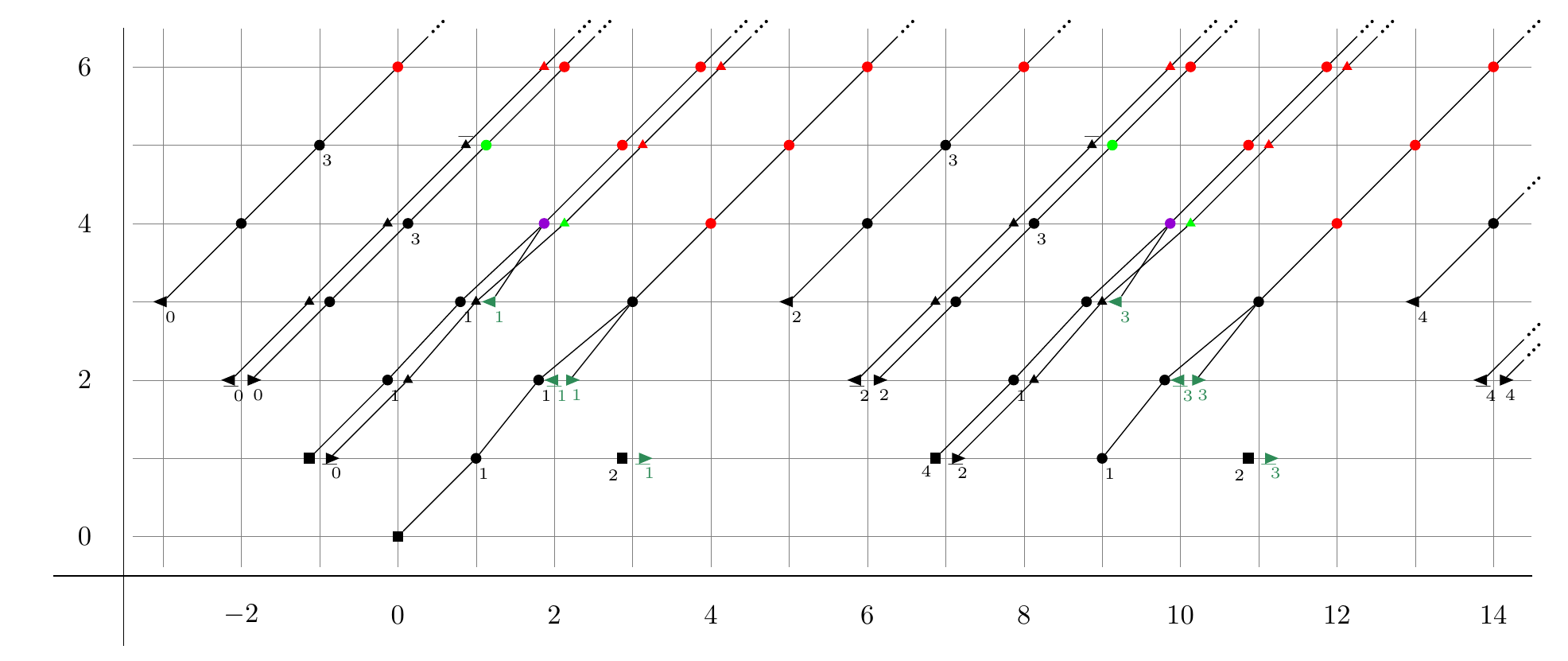}
\caption{The $E_\infty$-page of the ESSS for $L$ over $\q_2$. A bullet $\bullet$ represents $\z/2[\tau]$, a red bullet \textcolor{red}{$\bullet$} represents $\z/2$, a green bullet \textcolor{green}{$\bullet$} represents $\z/2[\tau^4]$, a bullet with subscript $\bullet_n$ represents $\z/2[\tau^4]\{1,\tau^n\}$, a lavender bullet \hana{$\bullet$} represents $\z/2[\tau]\{\tau^2\}\oplus \z/2$,
a triangle $\blacktriangle$ represents $(\z/2[\tau])^3$, a red triangle $\color{red} \blacktriangle$ represents $(\z/2)^3$, a overlined triangle $\overline \blacktriangle$ represents $\z/2[\tau^2]\oplus \z/2[\tau^2]\{\tau^3\}\oplus (\z/2)^2$, a green triangle $\color{green} \blacktriangle$ represents $\z/2[\tau^2]\oplus \z/2[\tau]\oplus \z/2[\tau^4]\{1,\tau,\tau^3\}$,
a square $\blacksquare$ represents $\z$, a square with the positive integer $n$ as its left subscript ${}_n\blacksquare$ represents $\z/2^n$, 
an isosceles triangle with apex pointing left and nonnegative integer $n$ as its right subscript $\blacktriangleleft_n ({\color{seaolive}\blacktriangleleft_n}, \underline{\blacktriangleleft_n}, {\color{seaolive}\underline \blacktriangleleft_n})$ represents $C_2(n)$ ($\tilde C_2(n)$, $\underline C_2(n)$, $\tilde {\underline C_2}(n)$), 
and an isosceles triangle with apex pointing right and nonnegative integer $n$ as its right subscript $\blacktriangleright_n ({\color{seaolive} \blacktriangleright_n}, {\underline \blacktriangleright_n}, {\color{seaolive}\underline \blacktriangleright_n})$ represents ${K}_2(n)$ ($\tilde { K_2}(n)$, ${\underline K_2}(n)$, $\tilde {\underline K_2}(n)$). 
}
\label{Fig:Q2E2L}
\end{figure}

\clearpage

\section{Tables}\label{Sec:Tables}

In this section, we provide tables describing the homotopy groups of $L$ over certain base fields. We emphasize that since $L$ is not known to be a ring spectrum, the tables only give additive information. In particular, we note that $\pi_{**}^{\r}(L)$ is only additively isomorphic to the underlying bigraded abelian group of the bigraded ring described in the relevant table. 

\subsection{Algebraically closed fields}



\begin{longtable}{llll}
\caption{Additive generators for $\pi_{**}^{\c}(L)$: $i,j,k \geq 0$. \label{Table:CLGens}} \\
\hline
Additive generator & Degree & Constraints & Degree of $\mathsf{h}$-torsion \\
\hline \endfirsthead
\caption[]{Additive generators for $\pi_{**}^{\c}(L)$} \\
\hline
Additive generator & Degree & Constraints & $\mathsf{h}$-torsion \\
\hline \endhead
\hline \endfoot
$\tau^i$ & $(0,-i)$ & & $\infty$ \\
$\iota v_1^{4k}h_1^j \tau^i$ & $(8k+j-1,4k+j-i)$ & $j \leq 2$ if $i >0$ & $\infty$ if $j=k=0$; \\
& & & $\nu(k)+4$ if $j=0$ and $k>0$; \\
& & &  $1$ if $j \geq 1$ \\ 
$\iota v_1^{4k+2} \tau^i$ & $(8k+3,4k+2-i)$ & & $3$ \\
$v_1^{4k} h_1^{j+1} \tau^i$ & $(8k+j+1, 4k+j+1-i)$ & $j \leq 2$ if $i >0$ & $1$
\end{longtable}

There is one additive relation between these generators:
$$4 \cdot \iota v_1^{4k+2} \tau^i = v_1^{4k} h_1^3 \tau^{i+1}.$$

\subsection{Finite fields}

\begin{longtable}{llll}
\caption{Additive generators for $\pi_{**}^{\f_q}(L)$: $i,j,k \geq 0$. \label{Table:FqLGens}} \\
\hline
Additive generator & Degree & Constraints & Degree of $\mathsf{h}$-torsion \\
\hline \endfirsthead
\caption[]{Additive generators for $\pi_{**}^{\f_q}(L)$} \\
\hline
Additive generator & Degree & Constraints & $\mathsf{h}$-torsion \\
\hline \endhead
\hline \endfoot
$1$ & $(0,0)$ & & $\infty$ \\
$\iota v_1^{2k}$ & $(4k-1,2k)$ & & $\infty$ if $k=0$ \\
& & &  $\nu(k)+3$ if $k>0$ \\
$\iota v_1^{4k} h_1^{j+1} \tau^i$ & $(8k+j,4k+j-1-i)$ & $i \neq 1$ if $j=2$; & 1 \\
& & $i=0$ if $j \geq 3$ & \\
$x_q \tau^i$ & $(-1,-i)$ & & $s_q(i)$ \\
$2 x_q v_1^{2k}\tau^i$ & $(4k-1,2k-1-i)$ & $2 \leq s_q(i) \leq \nu(k)+3$ & $s_q(i)$\\
$2^{s_q(i) - \nu(k)-3} x_q v_q^{2k} \tau^i$ & $(4k-1,2k-1-i)$ & $s_q(i)>3$ & $\nu(k)+3$ \\
$x_q v_1^{4k} h_q^{j+1} \tau^i$ & $(8k+j,4k+j-i)$ & $i=0$ or $s_q(i-1)>3$ if $j=2$; & $1$ \\
& & $i=0$ if $j \geq 3$ & \\
$\iota x_q h_1^j \tau^i$ & $(j-2,j-1-i)$ & $i=0$ if $j \geq 3$ & $s_q(i)$ if $j=0$; \\
& & & $1$ if $j>0$ \\
$2 \iota x_q v_1^{4k} \tau^i$ & $(8k-2, 4k-1-i)$ & $k>0$ & $s_q(i)$ if $s_q(i) \leq \nu(k)+4$; \\
& & & $s_q(i)-\nu(k)-4$ else \\
$\iota x_q v_1^{4k+2} \tau^i$ & $(8k+2,4k+1-i)$ & & $s_q(i)$ if $s_q(i) \leq \nu(k)+4$; \\
& & & $s_q(i)-\nu(k)-4$ else \\
$\iota x_q v_1^{4k} h_1^{j+1} \tau^i$ & $(8k-1+j,4k+j-i)$ & $i=0$ if $j \geq 2$ & $1$
\end{longtable}

There are two additive relations between these generators:
$$4 \cdot \iota v_1^{4k+2} \tau^i = v_1^{4k} h_1^3 \tau^{i+1},$$
$$4 \cdot \iota x_q v_1^{4k+2} \tau^i = x_q v_1^{4k} h_1^3 \tau^{i+1},$$
which hold whenever the relevant elements are defined.

\subsection{$\q_q$ for $q$ odd}

The coefficients $\pi_{**}^{\q_q}(L)$ are determined by the additive isomorphism
$$\pi_{**}^{\q_q}(L) \cong \pi_{**}^{\f_q}(L)\{1,\pi\}.$$
A table of generators for $\q_q$, $q$ odd, can be obtained from \cref{Table:FqLGens} by adding a generator  `$\pi x$' for each additive generator $x$ in degree $deg(x)-(1,1)$ with the same constrains and $\mathsf{h}$-torsion degree as $x$.

\subsection{$\r$}
The cases for $\r$ is more complicated due to the $\rho$ multiples. 

\begin{longtable}{lllll}
\caption{Multiplicative generators for $\pi_{**}^{\r}(L)$: $k \geq 0$. \label{Table:RLGens}} \\
\hline
Name & Degree & $\mathsf{h}$-torsion & $\rho$-torsion &  $\eta$-torsion \\
\hline \endfirsthead
\\
\hline
Name & Degree & $\mathsf{h}$-torsion & $\rho$-torsion &  $\eta$-torsion \\
\hline \endhead
\hline \endfoot
coweight$\equiv 0 (4)$ &&&&\\
\hline
$\mathsf{h}$ & $(0,0)$ & $\infty$   & 1 & 1\\
$\rho$ & $(-1,-1)$ & 1 &   $\infty$ & $\infty$\\
 $\eta$ & $(1,1)$ & 1 &   $\infty$ & $\infty$\\
 $2\tau^{4j}, j\geq 1$ & $(0,-4j)$ & $\infty$& 1 & 1\\
 $\iota \tau^{1+4j} h_1 v_1^{4k}$ &$(8k,4k-4j)$ &2&2&2\\
 \hline
 coweight$\equiv 1 (4)$ &&&&\\
\hline
$\iota (\tau h_1)^2 v_1^{4k}\tau^{4j}$ &$(8k+1,4k-4j)$& 2&2&2\\
$\tau^{1+4j} h_1 v_1^{4k}$ &$(8k+1,4k-4j)$& 2&3&3\\
$\iota v_1^{4k} \cdot 2\tau^{2+4j}$&$(8k-1,4k-2-4j)$& $\nu(k)+3$&1&3\\
$\iota 2v_1^{4k+2}\tau^{4j}$&$(8k+3,4k+2-4j)$& 3&3&1\\
\hline 
coweight$\equiv 2 (4)$ &&&&\\
\hline
$(\tau h_1)^2 v_1^{4k}\tau^{4j}$ &$(8k+2,4k-4j)$& 2&2&2\\
$2\tau^{2+4j}$&$(0,-2-4j)$& $\infty$&1&3\\
\hline
 coweight$= -1 $ &&&&\\
 \hline
 $\iota$ &(-1,0)&$\infty$&$\infty$&$\infty$\\
 \hline 
 coweight$\equiv 2^{n-1}-1 (2^n), n\geq 3 $ &&&&\\
 \hline
   $(\tau h_1)^3 2\tau^{2^{n-1}-4k+2^nj-3}v_1^{4k}$ &$(8k+3,8k+4-2^{n-1}-2^nj)$&2&1&1\\
   $\iota \tau^{2^{n-1}-4k+2^nj}v_1^{4k}$ &$(8k-1,8k-2^{n-1}-2^nj)$
 &$\nu(k)+4$&&\\
 $\iota 2\tau^{2^{n-1}-4k+2^nj-2}v_1^{4k+2}$ &$(8k+3,8k+4-2^{n-1}-2^nj)$&4&&\\
\end{longtable}

We take $\nu(0)$ to be $\infty$.
The $\rho$-torsion and $\eta$-torsion of the last two lines are given by the following relations.

$\iota \tau^{2^{n-1}-4k+2^nj}\cdot (\eta\rho)^{n}\eta=0.$

$\iota \tau^{2^{n-1}-4k+2^nj}v_1^{4k}\cdot (\eta\rho)^{n}\rho=0.$

$\iota \tau^{2^{n-1}+2^nj}\cdot (\eta\rho)^n=0$.

$\iota 2\tau^{2^{n-1}-4k+2^nj-2}v_1^{4k+2} \cdot \rho = \iota \tau^{2^{n-1}-4k+2^nj}v_1^{4k}\cdot \eta^3$.

$\iota 2\tau^{2^{n-1}-4k+2^nj-2}v_1^{4k+2} \cdot \eta = \iota \tau^{2^{n-1}-4k-4+2^nj}v_1^{4k+4}\cdot \rho^3$.

$2\iota 2\tau^{2^{n-1}-4k+2^nj-2}v_1^{4k+2} \cdot \mathsf{h} = (\tau h_1)^3 2\tau^{2^{n-1}-4k+2^nj-3}v_1^{4k}$.

$(\tau h_1)^3 2\tau^{2^{n-1}-4k+2^nj-3}v_1^{4k}\cdot \mathsf{h} =
\iota \tau^{2^{n-1}-4k+2^nj}v_1^{4k}\cdot \eta^4\cdot (\eta\rho)^{n-1}
$.

$2\cdot (\tau h_1 v_1^{4k})=0$.

$\eta^2\rho\cdot (\tau h_1 v_1^{4k})=0$.

$\eta\rho^2\cdot (\tau h_1 v_1^{4k})=0$.

$\rho^2\cdot (\tau h_1)=\mathsf{h}\cdot \iota \cdot 2\tau^2=0$.

$\rho^2\cdot (\tau h_1)v_1^{4k}=\mathsf{h}\cdot \iota v_1^{4k}\cdot 2\tau^2$.

$\eta \cdot \iota v_1^{4k} \cdot 2\tau^2= \iota (\tau h_1)^2 v_1^{4k}\cdot \rho$.

$\eta^2\cdot (\tau h_1)v_1^{4k}=\mathsf{h}\cdot \iota 2v_1^{4k+2}$.

$\rho \cdot \iota v_1^{4k+2}= \iota (\tau h_1)^2 v_1^{4k}\cdot \eta$.

$2\tau^2\cdot \eta=\rho\cdot (\tau h_1)^2$.

\subsection{$\q_2$ and $\q$}

It is possible to produce tables of additive generators for $\pi_{**}^{\q_2}(L)$ and $\pi_{**}^{\q}(L)$, but they would be quite long and (in our view) not particularly enlightening. We have therefore chosen to omit them from this paper.

\bibliographystyle{alpha}
\bibliography{master}

\begin{thebibliography}{GHIR19}

\bibitem[AM17]{AM17}
Michael Andrews and Haynes Miller.
\newblock Inverting the {Hopf} map.
\newblock {\em J. Topol.}, 10(4):1145--1168, 2017.

\bibitem[AR{\O}20]{ARO17}
Alexey Ananyevskiy, Oliver R\"{o}ndigs, and Paul~Arne {\O}stv{\ae}r.
\newblock On very effective hermitian {$K$}-theory.
\newblock {\em Math. Z.}, 294(3-4):1021--1034, 2020.

\bibitem[Bac22]{Bac22}
Tom Bachmann.
\newblock {$\eta$}-periodic motivic stable homotopy theory over {D}edekind
  domains.
\newblock {\em J. Topol.}, 15(2):950--971, 2022.

\bibitem[BCQ23]{BCQ21}
William Balderrama, Dominic~Leon Culver, and J.D. Quigley.
\newblock The motivic lambda algebra and motivic {H}opf invariant one problem.
\newblock {\em Geometry \& Topology, to appear}, 2023.

\bibitem[BH20]{BH20}
Tom Bachmann and Michael~J Hopkins.
\newblock $\eta$-periodic motivic stable homotopy theory over fields.
\newblock {\em arXiv preprint arXiv:2005.06778}, 2020.

\bibitem[BI22]{BI22}
Eva Belmont and Daniel~C. Isaksen.
\newblock {$\Bbb R$}-motivic stable stems.
\newblock {\em J. Topol.}, 15(4):1755--1793, 2022.

\bibitem[BIK22]{BIK22}
Eva Belmont, Daniel~C Isaksen, and Hana~Jia Kong.
\newblock $\mathbb{R}$-motivic $v_1$-periodic homotopy.
\newblock {\em arXiv preprint arXiv:2204.05937}, 2022.

\bibitem[BK05]{BK05}
A.~J. Berrick and M.~Karoubi.
\newblock Hermitian {$K$}-theory of the integers.
\newblock {\em Amer. J. Math.}, 127(4):785--823, 2005.

\bibitem[Boa99]{Boa98}
J.~Michael Boardman.
\newblock Conditionally convergent spectral sequences.
\newblock In {\em Homotopy invariant algebraic structures ({B}altimore, {MD},
  1998)}, volume 239 of {\em Contemp. Math.}, pages 49--84. Amer. Math. Soc.,
  Providence, RI, 1999.

\bibitem[BOQ23]{BOQ23}
William Balderrama, Kyle Ormsby, and J.D. Quigley.
\newblock A motivic analogue of the ${K}(1)$-local sphere spectrum.
\newblock {\em Journal of the European Mathematical Society (JEMS), to appear},
  2023.

\bibitem[CQ21]{CQ21}
Dominic~Leon Culver and J.D. Quigley.
\newblock {$kq$}-resolutions {I}.
\newblock {\em Trans. Amer. Math. Soc.}, 374(7):4655--4710, 2021.

\bibitem[DI10]{DI10}
Daniel Dugger and Daniel~C. Isaksen.
\newblock The motivic {Adams} spectral sequence.
\newblock {\em Geom. Topol.}, 14(2):967--1014, 2010.

\bibitem[DI17a]{DI16a}
Daniel Dugger and Daniel Isaksen.
\newblock Low-dimensional {Milnor}-{Witt} stems over {{\(\mathbb R\)}}.
\newblock {\em Ann. \(K\)-Theory}, 2(2):175--210, 2017.

\bibitem[DI17b]{DI16}
Daniel Dugger and Daniel~C. Isaksen.
\newblock {{\(\mathbb {Z}/2\)}}-equivariant and {{\(\mathbb {R}\)}}-motivic
  stable stems.
\newblock {\em Proc. Am. Math. Soc.}, 145(8):3617--3627, 2017.

\bibitem[DM89]{DM89}
Donald~M Davis and Mark Mahowald.
\newblock The image of the stable {J}-homomorphism.
\newblock {\em Topology}, 28(1):39--58, 1989.

\bibitem[Fri76]{Fri76}
Eric~M. Friedlander.
\newblock Computations of {K}-theories of finite fields.
\newblock {\em Topology}, 15:87--109, 1976.

\bibitem[GHIR19]{GHIR19}
Bertrand~J Guillou, Michael~A Hill, Daniel~C Isaksen, and Douglas~Conner
  Ravenel.
\newblock The cohomology of ${C}_2$-equivariant ${A}(1)$ and the homotopy of
  $ko_{{C}_2}$.
\newblock {\em Tunisian Journal of Mathematics}, 2(3):567--632, 2019.

\bibitem[GI15]{GI15}
Bertrand~J. Guillou and Daniel~C. Isaksen.
\newblock The {{\(\eta\)}}-local motivic sphere.
\newblock {\em J. Pure Appl. Algebra}, 219(10):4728--4756, 2015.

\bibitem[GI16]{GI16}
Bertrand~J. Guillou and Daniel~C. Isaksen.
\newblock The {{\(\eta\)}}-inverted {{\(\mathbb{R}\)}}-motivic sphere.
\newblock {\em Algebr. Geom. Topol.}, 16(5):3005--3027, 2016.

\bibitem[Hil11]{Hil11}
Michael~A. Hill.
\newblock Ext and the motivic {Steenrod} algebra over {{\(\mathbb R\)}}.
\newblock {\em J. Pure Appl. Algebra}, 215(5):715--727, 2011.

\bibitem[HKO11]{HKO11a}
P.~Hu, I.~Kriz, and K.~Ormsby.
\newblock Convergence of the motivic {Adams} spectral sequence.
\newblock {\em J. \(K\)-Theory}, 7(3):573--596, 2011.

\bibitem[I{\O}20]{IO18}
Daniel~C. Isaksen and Paul~Arne {\O}stv{\ae}r.
\newblock Motivic stable homotopy groups.
\newblock In {\em Handbook of homotopy theory}, CRC Press/Chapman Hall Handb.
  Math. Ser., pages 757--791. CRC Press, Boca Raton, FL, [2020] \copyright
  2020.

\bibitem[IS11]{IS11}
Daniel~C Isaksen and Armira Shkembi.
\newblock Motivic connective {K}-theories and the cohomology of {A}(1).
\newblock {\em Journal of K-theory: K-theory and its Applications to Algebra,
  Geometry, and Topology}, 7(03):619--661, 2011.

\bibitem[Isa19]{Isa19}
Daniel~C. Isaksen.
\newblock Stable stems.
\newblock {\em Mem. Amer. Math. Soc.}, 262(1269):viii+159, 2019.

\bibitem[IWX20]{IWX20}
Daniel~C Isaksen, Guozhen Wang, and Zhouli Xu.
\newblock More stable stems.
\newblock {\em arXiv preprint arXiv:2001.04511}, 2020.

\bibitem[KR{\O}20]{KRO20}
Jonas~Irgens Kylling, Oliver R{\"o}ndigs, and Paul~Arne {\O}stv{\ae}r.
\newblock Hermitian {{\(K\)}}-theory, {D}edekind {{\(\zeta \)}}-functions, and
  quadratic forms over rings of integers in number fields.
\newblock {\em Camb. J. Math.}, 8(3):505--607, 2020.

\bibitem[Kyl15]{Kyl15}
Jonas~Irgens Kylling.
\newblock Hermitian {K}-theory of finite fields via the motivic {A}dams
  spectral sequence.
\newblock Master's thesis, University of Oslo, Norway, 2015.

\bibitem[Lev08]{Lev08}
Marc Levine.
\newblock The homotopy coniveau tower.
\newblock {\em J. Topol.}, 1(1):217--267, 2008.

\bibitem[Mah81]{Mah81}
Mark Mahowald.
\newblock bo-{R}esolutions.
\newblock {\em Pacific Journal of Mathematics}, 92(2):365--383, 1981.

\bibitem[Mil70]{Mil69}
John Milnor.
\newblock Algebraic {$K$}-theory and quadratic forms.
\newblock {\em Invent. Math.}, 9:318--344, 1969/70.

\bibitem[Mor12]{Mor12}
Fabien Morel.
\newblock {\em {$\Bbb A^1$}-algebraic topology over a field}, volume 2052 of
  {\em Lecture Notes in Mathematics}.
\newblock Springer, Heidelberg, 2012.

\bibitem[MV99]{MV99}
Fabien Morel and Vladimir Voevodsky.
\newblock {{\(\mathbb{A}^1\)}}-homotopy theory of schemes.
\newblock {\em Publ. Math., Inst. Hautes {\'E}tud. Sci.}, 90:45--143, 1999.

\bibitem[O{\O}13]{OO13}
Kyle~M. Ormsby and Paul {\O}stv{\ae}r.
\newblock Motivic {Brown}-{Peterson} invariants of the rationals.
\newblock {\em Geom. Topol.}, 17(3):1671--1706, 2013.

\bibitem[O{\O}14]{OO14}
Kyle~M. Ormsby and Paul~Arne {\O}stv{\ae}r.
\newblock Stable motivic {$\pi_1$} of low-dimensional fields.
\newblock {\em Adv. Math.}, 265:97--131, 2014.

\bibitem[OR20]{OR20}
Kyle Ormsby and Oliver R\"{o}ndigs.
\newblock The homotopy groups of the {$\eta$}-periodic motivic sphere spectrum.
\newblock {\em Pacific J. Math.}, 306(2):679--697, 2020.

\bibitem[Orm11]{Orm11}
Kyle~M. Ormsby.
\newblock Motivic invariants of {$p$}-adic fields.
\newblock {\em J. K-Theory}, 7(3):597--618, 2011.

\bibitem[Qui21a]{Qui21c}
J.D. Quigley.
\newblock Motivic {M}ahowald invariants over general base fields.
\newblock {\em Doc. Math.}, 26:561--582, 2021.

\bibitem[Qui21b]{Qui21b}
J.D. Quigley.
\newblock Real motivic and {$C_2$}-equivariant {M}ahowald invariants.
\newblock {\em J. Topol.}, 14(2):369--418, 2021.

\bibitem[R{\O}16]{RO16}
Oliver R{\"o}ndigs and Paul {\O}stv{\ae}r.
\newblock Slices of {Hermitian} {{\(K\)}}-theory and {Milnor}'s conjecture on
  quadratic forms.
\newblock {\em Geom. Topol.}, 20(2):1157--1212, 2016.

\bibitem[RS{\O}19]{RSO19}
Oliver R{\"o}ndigs, Markus Spitzweck, and Paul {\O}stv{\ae}r.
\newblock The first stable homotopy groups of motivic spheres.
\newblock {\em Ann. Math. (2)}, 189(1):1--74, 2019.

\bibitem[RS{\O}21]{RSO21}
Oliver R{\"o}ndigs, Markus Spitzweck, and Paul~Arne {\O}stv{\ae}r.
\newblock The second stable homotopy groups of motivic spheres.
\newblock {\em arXiv preprint arXiv:2103.17116}, 2021.

\bibitem[RW00]{RW00}
J.~Rognes and C.~Weibel.
\newblock Two-primary algebraic {$K$}-theory of rings of integers in number
  fields.
\newblock {\em J. Amer. Math. Soc.}, 13(1):1--54, 2000.
\newblock Appendix A by Manfred Kolster.

\bibitem[Voe03]{Voe03}
Vladimir Voevodsky.
\newblock Reduced power operations in motivic cohomology.
\newblock {\em Publ. Math., Inst. Hautes {\'E}tud. Sci.}, 98:1--57, 2003.

\bibitem[Voe11]{Voe11}
Vladimir Voevodsky.
\newblock On motivic cohomology with {{\(\mathbb{Z}/l\)}}-coefficients.
\newblock {\em Ann. Math. (2)}, 174(1):401--438, 2011.

\bibitem[Wil16]{Wil16}
Glen~M. Wilson.
\newblock {\em Motivic stable stems over finite fields}.
\newblock ProQuest LLC, Ann Arbor, MI, 2016.
\newblock Thesis (Ph.D.)--Rutgers The State University of New Jersey - New
  Brunswick.

\bibitem[Wil18]{Wil18}
Glen~Matthew Wilson.
\newblock The eta-inverted sphere over the rationals.
\newblock {\em Algebr. Geom. Topol.}, 18(3):1857--1881, 2018.

\bibitem[W{\O}17]{WO17}
Glen~Matthew Wilson and Paul {\O}stv{\ae}r.
\newblock Two-complete stable motivic stems over finite fields.
\newblock {\em Algebr. Geom. Topol.}, 17(2):1059--1104, 2017.

\end{thebibliography}

\end{document}